\documentclass[11pt, english]{amsart}

\usepackage{times}
\usepackage[latin1]{inputenc}
\usepackage{babel}
\usepackage{amsopn}
\usepackage{amsmath,amsthm,amssymb}
\usepackage{enumerate}

\usepackage[pdftex]{hyperref}
\usepackage{epsfig}

\usepackage{color}

\newcommand{\nc}{\newcommand}

\renewcommand{\a}{{\alpha}}
\renewcommand{\b}{{\beta}}

\newcommand{\fa}{{\mathfrak a}}

\newcommand{\ff}{{\mathfrak f}}
\newcommand{\fg}{{\mathfrak g}}
\newcommand{\fh}{{\mathfrak h}}

\newcommand{\fk}{{\mathfrak k}}

\newcommand{\fn}{{\mathfrak n}}
\nc{\fo}{{\mathfrak o}}

\newcommand{\fs}{{\mathfrak s}}
\newcommand{\fu}{{\mathfrak u}}
\newcommand{\fv}{{\mathfrak v}}

\newcommand{\fz}{{\mathfrak z}}
\newcommand{\cc}{{\mathbb C}}
\nc{\aff}{\fa\ff\ff}

\nc{\nn}{{\mathbb N}}
\newcommand{\rr}{{\mathbb R}}
\newcommand{\RR}{{\mathbb R}}
\nc{\qq}{\mathbb Q} \nc{\zz}{{\mathbb Z}}
\nc{\ad}{\operatorname{ad}} \nc{\Aut}{\operatorname{Aut}}

\usepackage{color}

\newtheorem{teo}{Theorem}[section]
\newtheorem{lem}[teo]{Lemma}
\newtheorem{prop}[teo]{Proposition}
\newtheorem{corol}[teo]{Corollary}
\newtheorem{rem}[teo]{Remark}

\nc{\lp}{\langle} \nc{\rp}{\rangle} \nc{\inc}{\hookrightarrow}

\title[ ]{ Classification of abelian complex structures on \\ 6-dimensional  Lie algebras}
\author{A. Andrada}
\date{}
\address{FaMAF-CIEM, Universidad Nacional de C\'ordoba, Ciudad Universitaria, 5000 C\'ordoba,
Argentina}
\email{andrada@famaf.unc.edu.ar}
\thanks{}

\author{M. L. Barberis}
\email{barberis@famaf.unc.edu.ar}
\author{I. G. Dotti}
\email{idotti@famaf.unc.edu.ar}


\begin{document}

\begin{abstract}
  We classify
the $6$-dimensional Lie algebras  that can be endowed with an
abelian complex structure and  parameterize, on each of these
algebras, the space of such structures up to holomorphic
isomorphism.\end{abstract}

\maketitle

\section{Introduction}

 Let $\fg$ be a Lie algebra, $J$  an
endomorphism of $\fg$ such that $J^2=-I$, and let ${\fg}^{1,0}$
 be the $ i$-eigenspace of $J$ in $\fg ^{\cc}
:= {\fg}\otimes_\rr \cc$. When ${\fg}^{1,0}$  is a complex
subalgebra we say that $J$ is integrable, when ${\fg}^{1,0}$
 is abelian we say that $J$ is  abelian and
when ${\fg}^{1,0}$  is a complex ideal we say that $J$ is
bi-invariant.    We note that a complex structure on a Lie algebra cannot be both abelian and bi-invariant, unless the Lie bracket is trivial.
If $G$ is a connected Lie group with Lie algebra
$\fg$, by left translating $J$ one obtains a complex manifold
$(G,J)$ such that left multiplication is holomorphic and, in the
bi-invariant case, also right multiplication is holomorphic, which
implies that $(G,J)$ is a complex Lie group.

Our concern here will be the case when $J$ is abelian. In
this case the  Lie algebra has abelian commutator, thus, it is
2-step solvable (see \cite{Pe}).  However, its nilradical need not
be abelian (see Remark \ref{nilradical}). Abelian complex
structures have interesting applications in hyper-K\"ahler with
torsion geometry (see \cite{BDV}). It has been shown in \cite{CF}
that the Dolbeault cohomology of a nilmanifold with an abelian
complex structure can be computed algebraically. Also,
deformations of abelian complex structures on nilmanifolds have
been studied in \cite{cfp}.

Of importance, when studying complex structures on a Lie algebra $\fg$, is the ideal $\fg'_J := \fg ' +J
\fg '$ constructed from algebraic and complex data. We will say that
the complex structure $J$ is proper when  $\fg'_J $ is properly
contained in~$\fg.$ Any complex structure on a nilpotent Lie algebra is proper \cite{ss}.
 The $6$-dimensional nilpotent Lie algebras carrying complex structures were classified in \cite{ss}, and those carrying abelian complex structures were classified in \cite{cfu}.

There is only one $2$-dimensional non-abelian Lie algebra,  the Lie algebra of the affine motion group of $\rr$, denoted by
$\aff(\rr)$. It
carries a unique  complex structure, up to
equivalence, which turns out to be abelian.  The $4$-dimensional Lie algebras admitting abelian
complex structures were classified in \cite{Sn}.  Each of these Lie algebras, with the
exception of $\aff(\cc)$, the realification of the Lie algebra of
the affine motion group of $\cc$, has a unique abelian complex structure up to equivalence.
 On $\aff(\cc)$ there  is a two-sphere of
abelian complex structures, but only two equivalence classes
distinguished by $J$ being proper or not. Furthermore, $\aff(\cc)$
is equipped with a natural  bi-invariant complex structure. In dimension~$6$ it turns out that, as a consequence of our results,  some of the Lie algebras
equipped with abelian complex structures are of  the form
$\aff(A)$ where $A$ is a $3$-dimensional commutative associative algebra.

In this article we obtain
the classification of the $6$-dimensional Lie algebras  that can
be endowed with an abelian complex structure. Moreover, on each of
the resulting Lie algebras, we parameterize the space of such
structures up to holomorphic isomorphism. It turns out that there are three nilpotent Lie algebras carrying curves of non-equivalent structures, while for the remaining Lie algebras the moduli space is finite.

 The outline of the paper is as follows. In \S\ref{prelim} we give the basic definitions,
recall  known results and study some properties of the Lie algebra
$\aff(\cc)$.  In addition, we parameterize the space of equivalence classes of complex structures on a Lie algebra with $1$-dimensional commutator (see Proposition~\ref{s'=1}).
In \S\ref{sec_six_nilpot} we   perform the
classification in the $6$-dimensional nilpotent case. This is
 done in two steps. Firstly, we obtain the
classification of the nilpotent Lie algebras admitting such a structure (see
Theorem \ref{12345}), recovering, by a different approach, the
results obtained in \cite{cfu,ss}. In the second step, we
parameterize, up to equivalence, the space of abelian complex
structures  on each of these algebras (see Theorem \ref{todo}).
As a consequence, we obtain  three nilpotent Lie algebras carrying curves of
non-equivalent abelian complex structures, namely, $\fh_3 \times
\fh_3$, $\fh_3(\cc)$ and a $3$-step nilpotent Lie algebra, where
$\fh_3$  is the $3$-dimensional real
Heisenberg Lie algebra and $\fh_3(\cc)$ is the realification of the
$3$-dimensional complex Heisenberg Lie algebra.

In \S\ref{class} we determine the $6$-dimensional non-nilpotent
Lie algebras $\fs$ with an abelian complex structure, studying in
Theorems \ref{teo_dim2}, \ref{non_abelian} and \ref{s'J=r^4}  the
case when such a structure is proper. The Lie
algebras, a finite number,   appearing in Theorems \ref{teo_dim2} and
\ref{non_abelian} are decomposable as a direct product of lower
dimensional Lie algebras with the corresponding abelian complex
structure preserving this decomposition.  In Theorem \ref{s'J=r^4},
 when considering the case $\fs'_J = \rr^4 $,  we obtain a $2$-parameter family of solvable Lie algebras carrying
both an abelian and a bi-invariant complex structure. Moreover, as a consequence of this result and Theorems \ref{teo_dim2} and
\ref{non_abelian}, it follows that the Lie algebra $\aff(\cc) \times \rr^2$ has exactly three abelian complex structures, up to equivalence, distinguished by the ideal $\fs'_J$.
Finally, in
Theorem \ref{non-proper} we consider the case when $J$ is not
proper  and obtain five isomorphism classes of Lie algebras, all of which are of the form
$\aff(A)$ where $A$ is a commutative associative algebra with
identity.

\

\noindent
{\sl Acknowledgement.} The authors would like to thank Professor J. Lauret for helpful observations.  They are also thankful to E. Rodriguez Valencia, who pointed out a mistake in a previous version of Theorem 3.4. 

\

\section{Preliminaries}\label{prelim}

\subsection{Complex structures on Lie algebras}\label{secaff}

A complex structure on a real Lie algebra $\fg$ is an endomorphism
$J$ of $\fg$  satisfying $J^2=-I$ and such that
\begin{equation}\label{nijen}
 J[x,y]-[Jx,y]-[x,Jy]-J[Jx,Jy]=0, \; \;\; \forall\, x,y
\in \fg.    \end{equation} It is well known that \eqref{nijen}
holds if  and only if ${\fg}^{1,0}$, the $i$-eigenspace of $J$, is
a complex subalgebra of $\fg ^{\cc} := {\fg}\otimes_\rr \cc$. A
pair $(\fg, J)$ will denote a Lie algebra $\fg$ equipped with a
complex structure $J.$

A complex structure $J$ on $\fg$ is called \textit{bi-invariant} if
the following condition holds:
\begin{equation}\label{bi}
J[x,y]=[Jx,y], \qquad  x,y \in \fg. \end{equation} In this case,
$(\fg , J)$ is a complex Lie algebra, that is, $\fg$ is turned into
a $\cc$-vector space by setting \[ (a+ib) x= ax +bJx, \qquad x\in
\fg, \; a,b \in \rr ,
\]  and condition \eqref{bi} ensures that the  Lie bracket is now
$\cc$-bilinear.

Two complex structures $J_1$ and $J_2$ on $\fg$ are said to be
\textit{equivalent} if there exists an automorphism $\alpha$ of
$\fg$ satisfying $J_2\,  \alpha = \alpha\, J_1.$ Two pairs $(\fg_1,
J_1)$ and $(\fg_2, J_2)$ are \textit{holomorphically isomorphic}
if there exists a Lie algebra isomorphism $\alpha:\fg_1 \to \fg_2$
satisfying $J_2 \, \alpha = \alpha \, J_1.$

Consider a Lie algebra $\fg$ equipped with a complex structure $J$.
We denote with $ \fg'$ the commutator ideal $[\fg ,\fg]$ of $\fg$
and with $\fg'_J$ the $J$-stable ideal of $\fg$ defined by
$\fg'_J:=\fg'+J\fg'.$ We will say that $J$ is \textit{proper} if
$\fg'_J\varsubsetneq\fg$. We note that if $(\fg_1, J_1)$ and
$(\fg_2, J_2)$ are holomorphically isomorphic, then $J_1$ is proper
if and only if $J_2$ is.

The next proposition provides sufficient conditions for a complex
structure to be proper. We recall that a Lie algebra is called
\textit{perfect} if it coincides with its commutator ideal.

\begin{prop}\label{strict}
\begin{enumerate}
\item Every complex structure on a nilpotent Lie algebra is
proper.
\item Every bi-invariant complex structure on a non-perfect Lie
algebra is proper.
\end{enumerate}
\end{prop}
\begin{proof} The first statement was proved in \cite[Theorem
1.3]{ss} (see also Lemma \ref{(O3)} below). The second assertion
follows since \eqref{bi} implies that the commutator ideal is
stable by any bi-invariant complex structure.
\end{proof}

\smallskip

\subsection{Abelian complex structures on Lie algebras}\label{sec-abel-prelim}

An {\it abelian} complex structure on $\fg$ is an endomorphism $J$
of $\fg$ satisfying
\begin{equation} J^2=-I, \hspace{1.5cm} [Jx,Jy]=[x,y], \; \;\;
\forall x,y \in \fg. \label{abel1}
\end{equation}
It follows that \eqref{abel1} is a particular case of
\eqref{nijen}; moreover, condition \eqref{abel1} is equivalent to
$\fg^{1,0}$ being abelian. These structures were first considered
in \cite{bdm}. A class of Lie  algebras carrying such structures
appears in the next result.

Let $\fh_{2n+1}=\text{span} \{ e_1, \dots , e_{2n}, z_0\}$ denote the $(2n
+1)$-dimensional Heisenberg algebra with Lie bracket   $[e_{2i-1},
e_{2i}]=z_0, \; 1\leq i\leq n,$ and
 $\fa \ff \ff (\RR)$  the
 Lie algebra of the group of affine motions of~$\rr$.

\begin{prop} \label{s'=1} If $\fg$ is an even dimensional real
Lie algebra with $1$-dimensional commutator~$\fg'$, then:
\begin{enumerate}
\item $\fg$ is isomorphic to either $\fh_{2n+1} \times \RR ^{2k+1} \text{ or } \aff(\RR) \times \RR
^{2k}$;
\item All these Lie algebras carry abelian complex structures
and every complex structure on $\fg$ is abelian;
\item There are $\left[\frac n2\right]+1$  equivalence classes of
complex structures on $\fh_{2n+1} \times \RR^{2k+1}$;
\item There is a unique complex structure on $\aff(\RR) \times \RR
^{2k}$ up to equivalence.
\end{enumerate}
\end{prop}
\begin{proof}
(1) follows from the characterization of the Lie algebras with one
dimensional commutator (see for example \cite[Theorem~4.1]{bd}).

A proof of (2) can be found in \cite[Examples~3.3 and Proposition
3.4]{ba-do}.

In order to prove (3), we recall that according to
\cite[Proposition 3.6]{R} the complex structures on $\fh_{2n+1} \times
\RR^{2k+1}$ are given by: \begin{equation}\label{uno} J
e_{2i-1}=\pm e_{2i}, \qquad \quad J z_{2j}= z_{2j+1}, \quad \;  1\leq
i\leq n, \;\; 0\leq j\leq k,\end{equation} where $\{ e_1, \dots ,
e_{2n}, z_0 \}$ is a basis of $\fh_{2n+1}$ such that $[e_{2i-1},
e_{2i}]=z_0$ and $\{z_1, \dots z_{2k+1} \}$ is a basis of
$\RR^{2k+1}$. Any two complex structures with the same number
of minus signs are equivalent, by permuting the elements of the
basis. Let $J_{r}$, with $0\leq r \leq n$, be any complex
structure as in \eqref{uno} such that $r$ is the number of minus
signs. By a suitable permutation of the basis, one can show that
$J_{r}$ is equivalent to $J_{n-r}$, thus we may assume $0\leq
r\leq \left[\frac{n}{2}\right]$.    Furthermore, let $0\leq r, \, s
\leq \left[\frac{n}{2}\right]$, and assume that $J_r $ is equivalent to $J_s$, that is, there exists an automorphism $\phi$ of $\fh_{2n+1}
\times \RR^{2k+1}$  such that $\phi J_r =J_s \phi$. Since $\phi$ preserves the commutator, we have $\phi z_0= c z_0$ for some $c\neq 0$. Set
$\mathfrak v:= \text{span} \{ e_1, \dots , e_{2n}\}$ and for $\alpha=r,s$ we define a bilinear form $B_{\alpha}$ on $\mathfrak v$ by $[x,J_{\alpha}y]=B_{\alpha}(x,y) z_0, \; x, y \in \mathfrak v$. It turns out that $B_{\alpha}$ is symmetric (since $J_{\alpha}$ is abelian) and non-degenerate. Moreover, $B_s\left(\phi (x), \phi(y)\right)= c B_r (x,y)$. Therefore, $B_r$ and $B_s$ have the same signature if $c>0$ or opposite signatures if $c<0$. Since $r,s \leq \left[\frac{n}{2}\right]$,
and Sign$\, \left(B_{\alpha}\right)= \left(2\alpha , 2(n-\alpha) \right)$,
in both cases we must have $r=s$ and (3)
follows.\footnote{ The
classification for $n=1, \, k=0,$ was given in \cite{Sn} and the
case $n=2,\, k=0 $  was carried out in \cite{u}. }

The proof of (4) follows by observing that given a complex
structure $J$ on $\aff(\RR) \times \RR ^{2k}$, the ideal $\fg'_J=\fg'+J\fg'$,
isomorphic to $\aff(\RR)$, is $J$-stable and complementary to the center $\RR ^{2k}$.
 It is clear that $\aff(\RR)$ admits a unique complex
structure up to equivalence, and the assertion follows.
\end{proof}

\

We include some properties of abelian complex structures in the
following lemma, whose proof is straightforward.

\medskip

\begin{lem}\label{propiedades}
Let $J$ be an abelian complex structure on a Lie algebra $\fg$.
Then:
\begin{enumerate}[(i)]
\item the center $ \fz(\fg)$ of $\fg$ is $J$-stable;
\item for any $x\in\fg$, $\ad_{Jx}=-\ad_xJ$; in particular,
the family of inner derivations of $\fg$ is stable under right multiplication by $J$;
\item for any $J$-stable ideal $\fu$ of $\fg$, the kernel of the map $\fg\to{\rm End}(\fu)$,
$x\mapsto \ad_x{|_{\fu}}$, is $J$-stable.
\end{enumerate}
\end{lem}

\medskip

 Some known obstructions to the existence of abelian
complex structures are stated below.

\begin{enumerate}

\item[(O1)]   Let $\fg$ be a real Lie algebra
admitting an abelian complex structure. Then $\fg$ is 2-step
solvable~{\cite{Pe}}.

\item[(O2)] Let $\fg$ be a solvable Lie algebra such that $\fg'$ has
codimension $1$ in $\fg$ and $\dim \fg >2$. Then $\fg$ does not
admit abelian complex structures  {\cite{ba-do}}.

\end{enumerate}

\medskip

Another obstruction in the case of nilpotent Lie algebras is given in the next lemma. Recall that the descending central series of $\fg$ is defined by
$\fg ^0 = \fg ,  \; \fg ^i=[\fg , \fg ^{i-1}], \; i\geq 1$.
\begin{lem}\label{(O3)}
Let $\fg$ be a $k$-step nilpotent Lie algebra with an abelian complex structure $J$  and  set $\fg ^i_J:=  \fg ^i+J \fg ^i$. Then
$ \fg ^i_J \varsubsetneq \fg ^{i-1}_J$ for all $i\leq k$. In particular, if $\dim \fg = 2m$, $\fg$ is at most $m$-step nilpotent.
\end{lem}

\begin{proof} Clearly, $\{ 0\}= \fg ^k_J \varsubsetneq \fg ^{k-1}_J$. Assume that $ \fg ^i_J = \fg ^{i-1}_J$ for some $i<k$.
Then,  any $x\in \fg ^{i-1}$ can be written as $x= y+Jw, \; y,w \in \fg ^{i}$. The fact that $J$ is abelian implies that $[\fg , x]\subset \fg ^{i+1}$ for all $x\in \fg ^{i-1}$, that is, $\fg ^{i} \subset\fg ^{i+1}$, contradicting that $\fg$ is $k$-step nilpotent.
\end{proof}

\medskip

\subsection{Dimension less than or equal to four}\label{four}

There are only two $2$-dimensional Lie algebras: $\rr^2$ and $\aff(\rr)$. Both carry a unique abelian complex structure, up
to equivalence.

Let $\fg$ be a $4$-dimensional solvable Lie algebra with basis $\{e_1,e_2,e_3,e_4\}$.
According to \cite{Sn}, if $\fg$ admits an abelian complex
structure, it is isomorphic to one and only one of the following:
\begin{enumerate}
\item[ ]$\fg _1 = \rr^4$,
\item[ ]  $\fg _2= \fh_3 \times \rr$, $[e_1,e_2]=e_3$,
\item [ ] $\fg _3= \aff(\rr) \times \rr^2$, $[e_1,e_2]=e_2$
\item[ ] $\fg _4= \aff(\rr) \times \fa
\ff\ff(\rr)$, $[e_1,e_2]=e_2,\, [e_3,e_4]=e_4$
\item[ ] $\fg_5= \aff(\rr) \ltimes _{\ad} \rr^2$,
$[e_1,e_2]=e_2,\, [e_1,e_4]=e_4,\, [e_2,e_3]=e_4$.
\item[ ]  $\fg _6= \aff(\cc)$, $[e_1,e_3]=e_3,\, [e_1,e_4]=e_4,\, [e_2,e_3]=e_4,\,
[e_2,e_4]=-e_3$,
\end{enumerate}
where $\fh_3$ is the 3-dimensional Heisenberg algebra and
$\aff(\cc)$ is the realification of the Lie algebra of the group of affine motions of the complex line.
\medskip

\begin{teo}[\cite{Sn}] \label{dim4} Let $\fg$ be a $4$-dimensional solvable Lie algebra
with an abelian complex structure $J$. Then $(\fg, J)$ is
holomorphically isomorphic to one and only one of the following:
\begin{enumerate}
\item[] $\fg_1:$ $Je_1=e_2,\, Je_3=e_4$,
\item[] $\fg_2:$ $Je_1=e_2,\, Je_3=e_4$,
\item[] $\fg_3:$ $Je_1=e_2,\, Je_3=e_4$,
\item[] $\fg_4:$ $Je_1=e_2,\, Je_3=e_4$,
\item[] $\fg_5:$ $Je_1=e_2,\, Je_3=-e_4$,
\item[] $\fg_6:$ $J_1e_1=-e_2,\, J_1e_3=e_4$,
\item[] $\fg_6:$ $J_2e_1=e_3,\, J_2e_2=e_4$.
\end{enumerate}
\end{teo}

\medskip

We include below a proof of the fact that, up to equivalence,
$\aff(\cc)$ has two abelian complex structures (see also
\cite{Sn}).
We recall that $\aff(\cc)$ has a basis $\{e_1,e_2,e_3,e_4\}$ with
the following Lie bracket
\begin{equation}\label{corch_aff(C)}
[e_1,e_3]=e_3,\quad [e_1,e_4]=e_4, \quad [e_2,e_3]=e_4,\quad
[e_2,e_4]=-e_3, \end{equation} and with automorphism group given by
(see for instance \cite{CPD}):
\begin{equation}\label{aut_aff(C)} \text{Aut}(\aff(\cc))=\left\{ \begin{pmatrix}
      1 & 0 & 0 & 0 \\
      0 & \epsilon & 0 & 0 \\
      c & -\epsilon d & a & -\epsilon b \\
      d & \epsilon c & b & \epsilon a
     \end{pmatrix} \; :\;  a, b, c,d, \in \rr, \;\epsilon =\pm 1 \; \right\}.
\end{equation}

\begin{prop}\label{1y2}
Any abelian complex structure on the Lie algebra $\aff(\cc)$ is
equivalent to either $J_1$ or $J_2$, where
\begin{eqnarray*}J_1e_1&=&-e_2,\qquad J_1e_3=e_4, \\
 J_2e_1&=&e_3,\qquad \; \; \; J_2e_2=e_4.\end{eqnarray*}
Moreover, $J_1$ is proper and $J_2$ is not, hence they are not
equivalent.
\end{prop}

\begin{proof}
Let us assume that $J$ is an abelian complex structure on $\fg
=\aff(\cc)$. Since $\dim \fg'=2$, there are two possibilities for
$\fg'_J$:

\medskip

\noindent $\bullet$ $\fg'_J=\fg'$. In this case, $\fg'$ is a
$J$-stable ideal of $\aff(\cc)$. Let
\[ J=\begin{pmatrix}
      A & 0 \\ B & C
     \end{pmatrix}
\]
be the matrix representation of $J$ with respect to the ordered
basis $\{e_1,e_2,e_3,e_4\}$. Since $J^2=-I$,  we obtain that
$A^2=-I$ and $C^2=-I$, and therefore $J$ can be written as
\[ J=\begin{pmatrix}
      \alpha & \gamma & 0 & 0 \\
      \beta & -\alpha & 0 & 0 \\
      a & c & e & g \\
      b & d & f & -e
     \end{pmatrix}, \]
with
\begin{equation}\label{alpha} \alpha^2+\beta\gamma=-1,\quad e^2+fg=-1, \end{equation}
and certain constraints on $a,b,c,d$. From the fact  that $J$
is abelian, together with \eqref{alpha}, we obtain that
\[ \alpha=e=0, \quad \beta=\pm 1, \quad f=-g=\gamma=-\frac{1}{\beta}. \]
Using again that $J^2=-I$, we arrive at $c=-b, \,d=a$, thus,
$J=J^+_{(a,b)}$ or $J=J^-_{(a,b)}$, where
\[ J^+_{(a,b)}:=\begin{pmatrix}
      0 & 1 & 0 & 0 \\
      -1 & 0 & 0 & 0 \\
      a & -b & 0 & -1 \\
      b & a & 1 & 0
     \end{pmatrix},\quad  J^-_{(a,b)}:=\begin{pmatrix}
      0 & -1 & 0 & 0 \\
      1 & 0 & 0 & 0 \\
      a & -b & 0 & 1 \\
      b & a & -1 & 0
     \end{pmatrix}. \]
Note that $J^+_{(0,0)}=J_1$. Consider now the automorphisms $\phi$
and $\psi$ of $\aff(\cc)$ given by
\[ \phi=\begin{pmatrix}
      1 & 0 & 0 & 0 \\
      0 & -1 & 0 & 0 \\
      \frac{b}{2} & \frac{a}{2} & 1 & 0 \\
      \frac{a}{2} & -\frac{b}{2} & 0 & -1
     \end{pmatrix}, \quad
\psi=\begin{pmatrix}
      1 & 0 & 0 & 0 \\
      0 & 1 & 0 & 0 \\
      -\frac{b}{2} & -\frac{a}{2} & 1 & 0 \\
      \frac{a}{2} & -\frac{b}{2} & 0 & 1
     \end{pmatrix}.
\]
It is easy to see that $\phi J^+_{(a,b)}=J_1\phi$ and $\psi
J^-_{(a,b)}=J_1\psi$, so that all the abelian complex structures
$J^+_{(a,b)}$ and  $J^-_{(a,b)}$ are equivalent to $J_1$. This
completes this case.

\

\noindent $\bullet$ $\fg'_J=\fg$. In this case $\fg = \fg ' \oplus
J\fg '$. Let \[ f_1= J e_3 =ae_1 +b e_2 +v, \qquad f_2=Je_4 =ce_1 +d
e_2 +w,\] $a,b,c,d \in \rr, \; v,w \in \text{span}\{e_3,e_4\}$.
Since $J$ is abelian, $[f_1,f_2]=0$ and $[f_2,e_3]=[f_1,e_4]$, hence
$c=-b, \; d=a$. Therefore, we have
\[\text{ad}_{f_1}|_{\fg '}= \begin{pmatrix} a & -b \\ b &a   \end{pmatrix},\qquad
\text{ad}_{f_2}|_{\fg '}= \begin{pmatrix} -b & -a \\ a &-b
\end{pmatrix}.\]
Observe that $a^2+b^2 \neq 0$, since otherwise $\fg$ would be
abelian. Setting
\[ \tilde f_1 = \frac1{a^2+b^2}(a f_1-bf_2),\qquad \tilde f_2 =
\frac1{a^2+b^2}(b f_1 +af_2), \] we obtain
\[\text{ad}_{\tilde f_1}|_{\fs '}= \begin{pmatrix} 1 & 0 \\ 0 &1  \end{pmatrix},\qquad
\text{ad}_{\tilde f_2}|_{\fs '}= \begin{pmatrix} 0 & -1 \\ 1 & 0
\end{pmatrix}.\]
The expression of $J$ with respect to the basis $\{ \tilde f_1 ,
\tilde f_2 , e_3 , e_4 \}$ is given by
\[    J=J_{(a,b)} = \begin{pmatrix}  & B \\ -B^{-1} & \end{pmatrix}, \quad \text{where} \quad
 B=\begin{pmatrix} -a & b \\ -b & -a \end{pmatrix}. \]
 The structure constants of $\fg$ with respect to the new basis $\{ \tilde f_1 ,
\tilde f_2 , e_3 , e_4 \}$ are given by \eqref{corch_aff(C)}. If
$\phi$ is the following automorphism of $\aff(\cc)$:
\[ \phi=\begin{pmatrix}
      1 & 0 &  &  \\
      0 & 1 &  &  \\
       &  & -a & b \\
       &  & -b & -a
     \end{pmatrix}, \]
it turns out that $\phi J_{(a,b)} = J_{(1,0)} \phi $, that is,
$J_{(a,b)}$ is equivalent to $J_{(1,0)}=J_2$,  and the proposition follows.
\end{proof}

Given $x=(x_1,x_2,x_3)\in S^2$, let $J_x:=x_1J_1+x_2J_2+x_3J_1J_2$
with $J_1$ and $J_2$ as above. Since $J_1J_2=-J_2J_1$, $J_x$ is an
abelian complex structure on $\aff(\cc)$ and Proposition \ref{1y2}
implies the following result:

\begin{corol} There is a two-sphere $J_x$, $x\in S^2$, of abelian
complex structures on $\aff(\cc)$. Moreover, $J_x$ is equivalent
to $J_1$ if and only if $x=(1,0,0)$ or $x=(-1,0,0)$. All the other
abelian complex structures  in this sphere are  equivalent to $J_2$.

\end{corol}

\begin{rem}
{\rm  
It was shown in \cite{ba} that $\aff(\cc)$ is the only
$4$-dimensional Lie algebra carrying an abelian hypercomplex
structure, that is, a pair of anticommuting abelian complex
structures. }
\end{rem}


\begin{rem}{\rm We point out that $\aff (\cc)$ admits a  bi-invariant
complex structure $J$, namely, \[  J e_1=e_2, \qquad Je_3=e_4 . \]
Moreover, any bi-invariant complex structure is given by $\pm J$ and $J$ is equivalent to $-J$ via an automorphism as in \eqref{aut_aff(C)} with $\epsilon =-1$. Therefore, $\aff (\cc)$ has a unique bi-invariant complex structure up to equivalence.}
\end{rem}

The next characterization of $(\aff (\cc), J_1)$ will be
frequently used in \S\ref{class}.

\begin{lem}\label{iso_aff}
Let $\fg$ be a $4$-dimensional Lie algebra spanned by
$\{f_1,\ldots,f_4\}$ with Lie bracket defined by:
\begin{eqnarray*}[f_1,f_3]&=&[f_2,f_4]=cf_3+df_4, \quad\quad [f_1,f_4]=
-[f_2,f_3]=-df_3+cf_4,
\\ {}[f_1,f_2]&=&xf_3+yf_4,
\end{eqnarray*}
with $c^2+d^2\neq 0$, and abelian complex structure $J$ given by
\[ Jf_1=f_2,\qquad\quad Jf_3=f_4.\]
 Then $(\fg, J)$ is holomorphically
isomorphic to $(\aff (\cc), J_1)$.
\end{lem}

\begin{proof}
Consider first
\[ \tilde{f}_1=\frac{cf_1+df_2}{c^2+d^2},\qquad\quad  \tilde{f}_2=J \tilde{f}_1=\frac{-df_1+cf_2}{c^2+d^2}.\]
In the new basis $\{\tilde{f}_1,\tilde{f}_2,f_3,f_4\}$, the Lie
bracket of $\fg$ becomes
\[
[\tilde{f}_1,\tilde{f}_2]=x'f_3+y'f_4,\qquad
[\tilde{f}_1,f_3]=[\tilde{f}_2,f_4]=f_3,\qquad
[\tilde{f}_1,f_4]=-[\tilde{f}_2,f_3]=f_4.
\]
Setting now \[ e_1:=\tilde{f}_1-\frac{\, y'}{2}f_3+\frac{\, x'}{2}f_4,\qquad
e_2:=-\tilde{f}_2+\frac{\, x'}{2}f_3+\frac{\, y'}{2}f_4,\qquad e_3:=f_3,\,e_4:=f_4,\]
we obtain a basis $\{e_1,\ldots,e_4\}$ of $\fg$ such that the Lie
bracket is given as in \eqref{corch_aff(C)} and the complex
structure $J$ coincides with $J_1$ from Proposition \ref{1y2}.
\end{proof}

\medskip

\section{Dimension six: the nilpotent case} \label{sec_six_nilpot}
In this section we classify, up to holomorphic
isomorphism, the $6$-dimensional nilpotent non-abelian Lie
algebras $\fn$  with abelian complex structures.

We begin by recalling from Proposition
\ref{strict} that any complex structure $J$ on a nilpotent Lie
algebra $\fn$ is proper, that is, $\fn_J' = \fn'+ J\fn'$ is a proper nilpotent ideal.  In
the $6$-dimensional case, as a consequence of the following general
result, it turns out that $\fn_J' $ is abelian.

\begin{prop}\label{s'_abel} If $J$ is an abelian
complex structure on a $6$-dimensional solvable Lie algebra $\fs$
such that $\fs_J'$ is nilpotent, then $\fs_J'$ is
abelian.
\end{prop}

\begin{proof}
If $\fs_J'=\fs$ then $\fs$ is nilpotent and Proposition \ref{strict}
implies that $\fs_J'\varsubsetneq \fs$, a contradiction. Therefore,
$\dim \fs_J'= 2$ or $4$. If $\dim \fs_J'= 2$ then $\fs'_J$ is
abelian.

We show next that if $\dim \fs_J'= 4$ and $\fs'_J$ is non-abelian
nilpotent, then we get a contradiction.  $J$ induces by restriction
an abelian complex structure on $\fs'_J$. The only $4$-dimensional
non-abelian nilpotent Lie algebra admitting abelian complex
structures  is $\fh_3 \times \rr$ and this complex structure is unique up to equivalence (Theorem \ref{dim4}). Therefore,
$\fs'_J\cong \fh_3 \times \rr$, and there exists a basis $e_1,\ldots,
e_4$ of $\fs'_J$ such that $[e_1,e_2]=e_3$ and $Je_1=e_2, \quad
Je_3=e_4$. Let span$\{f_1, f_2\}$ be a complementary subspace of
$\fs_J'$ such that $f_2=Jf_1$. Setting $D=\ad_{f_1}|_{\fs_J'}$, it
follows from Lemma \ref{propiedades}$(ii)$ that
$\ad_{f_2}|_{\fs_J'}=-DJ$. The Lie bracket $[f_1,f_2]$ lies in
$\fs'\subseteq \fs'_J$, hence, $[f_1,f_2]= a e_1 + be_2 + ce_3
+de_4$, with $a,b,c,d\in\rr$.

Since $D$ is a derivation of $\fh_3 \times \rr$, it is of the form
\[ D= \begin{pmatrix} A & 0  \\ B & C
\end{pmatrix} , \]
where $A,B,C $ are $2\times 2$ real matrices such that $c_{11}={\rm
tr}\, A$ and $c_{21}=0$. Moreover, $DJ$ is also a derivation of
$\fh_3 \times \rr$, and therefore $c_{12}=a_{12}-a_{21}$ and
$c_{22}=0$.  Observe that $\fs'_J$ is $2$-step solvable and
$\fu:=\text{span}\{e_3,e_4\}\subset \fs'_J$ is an abelian ideal containing $(\fs'_J)'$, therefore the Jacobi identity implies\footnote{
 Given a $2$-step solvable Lie algebra $\fg$ and an  abelian
ideal $\fu$ of $\fg$ such that $\fg'\subseteq\fu$, the Jacobi
identity implies that $ \{ \ad_x{|_{\fu}} :  x\in \fg\} $ is a
commutative family of endomorphisms of $\fu.$ }  \[ C^2J|_{\fu}=C(J|_{\fu})C,      \]
and it follows that $C=0$.
Therefore, $A$ takes the following form:
\[ A= \begin{pmatrix} a_{11} & a_{12}  \\ a_{12} & -a_{11}
\end{pmatrix} . \]
Since $\fs'$ is abelian, we have that ${\rm rank}\, A
\leq 1$, hence  $a_{11} = a_{12}=0$, that
is, $A=0$.
 Finally, the Jacobi identity applied to $f_1, f_2, e_j, \; j=1,2$, implies that
  $a=b=0$, that is,
 $[f_1,f_2]\in \fu$,  hence $\fs'_J=\fu$, a contradiction.
\end{proof}


In the remainder of this section we  perform the classification.
This will be done in two steps. In the first one, we obtain the
classification of the Lie algebras admitting such a structure (see
Theorem \ref{12345}), recovering  by a different approach the
results obtained in \cite{cfu,ss}. In the second step, we
parameterize the space of abelian complex structures up to
equivalence on each of these algebras (see Theorem \ref{todo}).

\medskip

 \begin{teo} \label{12345} Let $\fn$ be a $6$-dimensional  nilpotent Lie
algebra with an abelian complex structure $J$. Then  $\fn $ is
isomorphic to one and only one of the following Lie algebras:
\begin{enumerate}
\item[ ] $\fn_1:= \fh_3 \times \rr^3 \! : \; [e_1,e_2]=e_6$,
\item[ ] $\fn_2:= \fh_5 \times \rr \, : \;[e_1,e_2]=e_6 =[e_3,e_4]$,
\item[ ] $\fn_3:=\fh_3 \times \fh_3 :\; [e_1,e_2]=e_5, \quad [e_3,e_4]=e_6$,
\item[ ] $\fn_4:=\fh_3(\cc):
\; [e_1,e_3]=-[e_2,e_4]=e_5, \quad [e_1,e_4]= [e_2,e_3]=e_6 $,
\item[ ] $\fn_5: \;  [e_1,e_2]=e_5, \quad [e_1,e_4]=
[e_2,e_3]=e_6 $, \item[ ] $\fn_6 : \; [e_1,e_2]=e_5, \quad
[e_1,e_4]= [e_2,e_5]=e_6 $, \item[ ]   $\fn_7 : \; [e_1,e_2]=e_4,
\quad [e_1,e_3]=-[e_2,e_4]=e_5, \quad [e_1,e_4]= [e_2,e_3]=e_6 $.
\end{enumerate}
\end{teo}
\begin{proof}
According to Lemma \ref{(O3)}, $\fn$ is at most
$3$-step nilpotent. To carry out the classification, we will
consider separately the $k$-step nilpotent case for $k=2$ or $3$.

$\bullet$ $\fn$ is $2$-step nilpotent: We observe that $\dim \fn '
=1$ or $2$. In fact, the center $\fz$ is a $J$-stable proper
ideal, hence $\dim \fz =2$ or $4$. If $\dim \fz =4$, then $\dim
\fn ' =1$, while if $\dim \fz =2$, we have that $\dim \fn ' =1$ or
$2$ since $\fn ' \subset \fz$.

If $\dim \fn' =1$, according to \cite[Theorem 4.1]{bd}, $\fn$ is
isomorphic to $\fn_1$ or $\fn_2$.

If $\dim \fn' =2$, since $\fz$ is $J$-stable, then $\fn '= \fz$ and
we have that $\dim \fz =2$. Let $\fv$ be a $J$-stable complementary
subspace of $\fz$. There are two possibilities:
\begin{enumerate}
\item[$(i)$] there exists $x \in \fv$ such that ${\rm rank}\, \ad
_{x}=1$;
\item[$(ii)$]  ${\rm rank}\, \ad
_{x}=2$ for any $x \in \fv, \, x\neq 0$.
\end{enumerate}
We consider next each case separately.

$-$ Case $(i)$: Consider the $3$-dimensional subspace $W:=\ker
\left( \ad_{x}|_{\fv}\right)$. It follows immediately that $\dim (
W\cap JW )= 2$.

If $x\notin W\cap JW$, there exists $e_1 \in W\cap JW$ such that
$\{ e_1, Je_1, x, Jx \}$ is a basis of $\fv$. Since $J$ is
abelian, it follows that the only non-zero brackets are
$[e_1,Je_1], \; [x,Jx]$, which must be linearly independent.
Setting $e_3:=x, \; e_4:=Jx, \; e_5:=[e_1,Je_1], \; e_6:= [x,Jx]$,  we obtain that $\fn \cong \fn _3$ and $J=J_{s,t}$ is given by
\begin{equation}\label{n3} J_{s,t}\, e_1=e_2, \;
J_{s,t}\, e_3=e_4, \; J_{s,t}\, e_5=s e_5+ te_6, \end{equation}
for some $s,t\in\rr,\,t\neq 0$, since the center
$\fz=\text{span}\{e_5,e_6\}$ is $J$-stable.

If $x\in W\cap JW$, then $\{x,Jx\}$ is a basis of $W\cap JW$.
There exists $e_1 \in W$ such that $Je_1 \notin W$ and therefore,
$\{ e_1, Je_1, x, Jx \}$ is a basis of $\fv$. Since $J$ is
abelian, it follows that the only non-zero brackets are
$[e_1,Je_1]$ and $ [e_1,Jx]=-[Je_1,x]$, which must be linearly
independent. Setting $e_2:=Je_1,\, e_3:=-x, \, e_4:=Jx,
e_5:=[e_1,Je_1],\, e_6:=[e_1,Jx]$, we obtain that $\fn \cong \fn
_5$ and $J=J_{s,t}$ is given by
\begin{equation}\label{n5}  J_{s,t}\, e_1=e_2, \; J_{s,t}\, e_3=-e_4, \; J_{s,t}\, e_5=s e_5+te_6,\end{equation}
for some $s,t\in\rr$ with $t\neq 0$, since the center
$\fz=\text{span}\{e_5,e_6\}$ is $J$-stable.

\smallskip

$-$ Case $(ii)$: In this case, there always exists $e_1 \in \fv$ such that $\ker \left(
\ad_{e_1}|_{\fv}\right)$ is not $J$-stable. Indeed, if $\ker \left( \ad_{x}|_{\fv}\right)$ were
$J$-stable for any $x\in \fv$, then we would have $[x,Jx]=0$ for any $x\in\fg$. Therefore, for any $x,y\in\fg$,
\[ 0=[x+y,J(x+y)]= [x,Jy]+[y,Jx]=2[x,Jy], \]
and this implies that $\fg$ is abelian, which is a contradiction.

Let us denote $W:=\ker \left(\ad_{e_1}|_{\fv}\right)$, which is not $J$-stable. Then $\fv=W\oplus JW$ and we fix a basis $\{e_1,
y, Je_1, Jy \}$ of $\fv$ with $y\in W$. It follows that
$[e_1,y]=0=[J e_1,J y]$. A basis of $\fz=\fn '$ is given by $
\{[e_1,Je_1] , [e_1,Jy]\}$ since ${\rm rank}\, \ad _{e_1}=2$.
Moreover, $ [e_1,Jy]=[y,Je_1]$ since $J$ is abelian, and ${\rm
rank}\, \ad _{y}=2$ implies $[y,Jy]\neq 0$. Therefore, there exist
$a,b\in \rr, a^2+b^2\neq 0$ such that
\[     [y,Jy]   =a[e_1,Je_1] +b [e_1,Jy] . \]
Since ${\rm rank}\, \ad _{w}=2$ for  all $w \in W,\; w\neq 0$, it follows that $\det(\ad_{ce_1+y}|_{JW})\neq 0$ for any $c$.
We have that $\det(\ad_{ce_1+y}|_{JW})=c^2+cb-a$, with respect to the bases $\{Je_1, Jy\}, \;
\{[e_1,Je_1] , [e_1,Jy]\}$ of $JW$ and $ \fz$, respectively. Therefore, it must be $b^2+4a<0$, and
setting
\[  e_2:= \frac 1{\sqrt{-(b^2+4a)}}(-b e_1   +2y), \quad  e_3:=Je_1, \quad e_4:=Je_2, \quad e_5:=[e_1,e_3], \quad e_6:= [e_1,e_4],  \]
it turns out that \[ [e_1,e_3]=-[e_2,e_4]= e_5, \qquad \quad
[e_1,e_4]=[e_2,e_3]=e_6. \] Hence, $\fn \cong \fn _4$ and
$J=J_{s,t}$ is given by 
\begin{equation}\label{n4}
J_{s,t}\, e_1=e_3, \; J_{s,t}\, e_2=e_4, \; J_{s,t}\, e_5=s
e_5+ te_6,
\end{equation}
for some $s,t\in\rr$ with $t\neq 0$, since the center
$\fz=\text{span}\{e_5,e_6\}$ is $J$-stable.

\smallskip

$\bullet$ $\fn$ is $3$-step nilpotent: In this case, $\fn
^2=[\fn , \fn '] \neq 0$ and $\fn ^2 \subseteq \fz$. The center
$\fz$ being $J$-stable, it follows that $\fn ^2_J= \fn^2 +J\fn^2
\subseteq \fz$. Since $\dim \fz \neq 4$ (otherwise, $\dim \fn '
=1$) it follows that $\fn ^2_J= \fz$ and $\dim \fn ^2_J=2$. From Lemma \ref{(O3)},
$\fn^2_J \varsubsetneq \fn'_J$, thus $\dim \fn'_J=4$. Let $V_0$ be a
$J$-stable subspace complementary to $\fn'_J$ and fix a basis
$\{f_1, Jf_1=f_2  \} $ of $V_0$. Using that $\fn'_J$ is abelian
(Proposition \ref{s'_abel}), we obtain \[ \fn '=[V_0 \oplus \fn'_J ,
V_0 \oplus  \fn'_J ] = [V_0 , V_0] + [V_0, \fn'_J]  \subseteq [V_0 , V_0] +\fn ^2  ,\]
thus $[f_1,f_2] \notin \fz$, hence $J [f_1,f_2] \notin \fz$, which implies that $\fn'_J =V_1 \oplus \fn^2_J $, where
$V_1$ is the subspace spanned by $[f_1,f_2]$ and $J [f_1,f_2] $.
Note that $\fn = V_0 \oplus V_1 \oplus \fn^2_J $ and that for any $y\in
V_0, \; y \neq 0$, the images of $\ad _{y}: V_1\to \fn^2_J$ coincide
with $\fn^2$.

$-$ If $\dim \fn^2=1$, we may assume that $[f_2, [f_1,f_2]] \neq
0$. By making a suitable change of basis of $V_0$, we obtain furthermore that $[f_1, [f_1,f_2]]=0$.  Setting
\[ e_1:=f_1,\; e_2:=f_2, \; e_3:= J[f_2,
[f_1,f_2]] , \; e_4:=-J[f_1,f_2], \; e_5:=[f_1,f_2], \;
e_6:=[f_2, [f_1,f_2]],\] we obtain that $\fn\cong \fn_6$ and $J$ is given by
\begin{equation}\label{n6} J e_1=e_2, \; Je_3=-e_6, \; Je_4=e_5.
\end{equation}

\smallskip

$-$ If $\dim \fn^2=2$ the map $\ad _{f_1}: V_1\to
\fn^2_J$ is an isomorphism and setting
\[ e_1:=f_1,\; e_2:=f_2, \; e_3:= J[f_1,f_2] , \; e_4:=[f_1,f_2], \; e_5:=[f_1,J[f_1,f_2]], \; e_6:=[f_1, [f_1,f_2]]\]
one has a basis of
 $\fn_7$. Since $J$ is abelian, it leaves $\fz=\text{span}\{e_5,
e_6\}$ stable, hence it is given by $J=J_{s,t}$, with
\begin{equation}\label{n7}
 J_{s,t} e_1= e_2, \; J_{s,t} e_3=-e_4, \; J_{s,t} e_5=se_5+te_6
\end{equation}
for some $s,t\in\rr,\,t\neq 0$.

\end{proof}

\medskip

\begin{rem} \label{six}
{\rm It follows from Theorem \ref{12345}  that if $\fn$ is a
$6$-dimensional nilpotent non-abelian Lie algebra and
 $J$ is any abelian complex structure, then
\begin{itemize}
  \item  $\fn'_J\cong\rr^2$ if and only if
 $\fn = \fn_k, \; 1\leq k\leq 5$,
\item $\fn'_J\cong\rr^4$ if and only if
  $\fn=\fn_6$ or $\fn_7$. \end{itemize}}\end{rem}

\medskip

\subsection{Equivalence classes of complex structures}
In this subsection we
parameterize the equivalence classes of abelian complex structures on the Lie algebras appearing in Theorem \ref{12345}.

Let us begin with $\fn_1$ and $\fn_2$. According to \cite{u} (see
also Proposition \ref{s'=1}) it follows that, up to equivalence,
$\fn_1$ has a unique complex structure $J$, given by
\[ J e_1=e_2, \;J e_3=e_4, \; J e_5= e_6,\]
whereas $\fn_2$ has two, given by
\[ J_{\pm}\, e_1=e_2, \quad J_{\pm}\, e_3={\pm}\, e_4, \quad J_{\pm}\, e_5= e_6  .\]

For the Lie algebra $\fn_6$, it follows from the proof of Theorem \ref{12345} that it has a unique abelian complex structure, given by \eqref{n6}.

We will determine next the equivalence classes of abelian complex structures on the
remaining Lie algebras. To achieve this, we consider a nilpotent
Lie algebra $\fn$ and the following space:
\[ \mathcal{C}_a(\fn)=\{J:\fn\to\fn : J \text{ is an abelian
complex structure on } \fn\},\] which we suppose is non-empty.
$\mathcal{C}_a(\fn)$ is a closed subset of $GL(\fn)$. The group
$\Aut(\fn)$ acts on $\mathcal{C}_a(\fn)$ by conjugation, and the
quotient $\mathcal{C}_a(\fn)/\Aut(\fn)$ parameterizes the
equivalence classes of abelian complex structures on $\fn$.

For $\fn=\fn_3,\,\fn_4,\, \fn_5$ and $\fn_7$, it follows from equations \eqref{n3}, \eqref{n4}, \eqref{n5} and \eqref{n7} in the proof of Theorem \ref{12345} that $\mathcal{Z}:=\{(s,t)\in\rr^2:\, t\neq 0\} $ can be considered as a subspace of $ \mathcal{C}_a(\fn)$ such that $\Aut(\fn)\cdot \mathcal{Z}=\mathcal{C}_a(\fn)$. The induced topology on $\mathcal{Z}$ coincides with the usual topology from $\rr^2$. Therefore, we have a homeomorphism
\[ \mathcal{C}_a(\fn)/\Aut(\fn)\cong \mathcal{Z}/G,\]
where $G:=\{ \phi\in\Aut(\fn):\, \phi(\mathcal{Z})=\mathcal{Z}\}$, a closed Lie subgroup of $\Aut(\fn)$\footnote{It can be shown that if $\phi \in \Aut(\fn)$ satisfies $\phi \cdot (s,t)=(s',t')$ for some $(s,t), \;(s',t')\in \mathcal{Z}$, then $\phi \in G$.}.

Let  $J$ be an abelian complex structure on $\fn$ and $\fv$  a
$J$-stable complementary subspace of the center $\fz$ in $\fn$. It
can be shown that  $\Aut^0(\fn)$ acts transitively on the orbit of $J$ in $\mathcal{C}_a(\fn)$, that is,
\[
\Aut(\fn)J=\Aut^0(\fn)J,
\]
where $\Aut^0(\fn)=\{\phi\in\Aut(\fn):\, \phi(\fv)=\fv \}$.
Indeed, if $J'=\phi J \phi ^{-1}$ with $\phi=\begin{pmatrix} A & 0
\\ B & C \end{pmatrix}\in\Aut(\fn)$, then $\phi^0=\begin{pmatrix} A & 0
\\ 0 & C \end{pmatrix}\in\Aut^0(\fn)$ satisfies $J'=\phi^0 J (\phi^0) ^{-1}$.
This implies that
\[ \mathcal{C}_a(\fn)/\Aut(\fn)\cong \mathcal{Z}/G^0\]
for the Lie algebras $\fn_i,\, i=3,4,5,7$, where
$G^0=G\cap\Aut^0(\fn)$.

We will determine in what follows for each algebra $\fn_i,\,
i=3,4,5,7$, the  group $G^0$ and the orbit space $\mathcal{Z}/G^0$.

\

$\bullet$ $\fn_3$: In this case, the group $G^0$ is given by a
disjoint union $G^0=G^+ \dot \cup \, G^-$, where
\[ G^+=\left\{\begin{pmatrix} a & -b & & & & \\ b & a & & & & \\ & & c & -d & & \\ & & d & c & &\\ &&&& a^2+b^2 & \\ &&&&& c^2+d^2 \end{pmatrix} : \,
 \begin{array}{l} a^2+b^2\neq 0,\\ c^2+d^2\neq 0 \end{array} \right\}, \quad G^-=\begin{pmatrix} & I & & \\ I & & & \\ & & 0 & 1 \\ & & 1
& 0 \end{pmatrix}\cdot G^+. \]
Let $(s,t), \, (s',t')\in\mathcal{Z}$ such that
 $\phi\cdot(s,t)=(s',t')$, with $\phi\in G^0$. It can be shown that:

\smallskip

\begin{enumerate}[$(i)$]
\item  if $\phi\in G^+$, then $s'=s$ and $tt'>0$;
\item if $\phi\in G^-$, then $s'=-s$ and $tt'<0$.
\end{enumerate}

Let us fix now $(s,t)\in\mathcal{Z}$. Taking $\phi_+\in G^+$ with
$a=\sqrt{|t|},\, c=1$ and $b=d=0$, we obtain that $\phi_+ \cdot
(s,t)=(s,\frac{t}{|t|})$. Taking now $\phi_-\in G^-$ with $a=1,\,
c=\sqrt{1+s^2}$ and $b=d=0$, we obtain that $\phi_-\cdot(s,-1)=
(-s,1)$. To sum up, we have that
\[ J_{s,t}\text{ is equivalent to either } \begin{cases} J_{s,1}\qquad \text{if } t>0,\\
                              J_{-s,1}\quad \,\text{  if } t<0. \end{cases} \]
It follows from $(i)$ and $(ii)$ that $J_{s,1}$ and $J_{s',1}$ are
equivalent if and only if $s=s'$, hence
\[ \mathcal{C}_a(\fn_3)/\Aut(\fn_3)\cong \{ J_{s,1}:\, s\in \rr\}.\]

\smallskip

$\bullet$ $\fn_5$: In this case, the action of $G$ on
$\mathcal{Z}$ is transitive, since the automorphism $\phi$ of
$\fn_5$ given by
\[ \phi=\begin{pmatrix}
0 & -1 &  &  &  &  \\
1 & 0 &  &  & &  \\
0 & -\frac{s}{2} & 0 & \frac{1+s^2}{t} &  &  \\
-\frac{s}{2} & 0 & -\frac{1+s^2}{t} & 0 & &  \\
 &  &  &  & 1 & 0 \\
 &  &  &  & -s & \frac{1+s^2}{t} \\
\end{pmatrix}\] satisfies $\phi\cdot (s,t)=(0,1)$ for any $(s,t)\in
\mathcal{Z}$. Thus,
$\mathcal{C}_a(\fn_5)/\Aut(\fn_5)=\{J_{(0,1)}\}$.

\smallskip

$\bullet$ $\fn_7$:  In this case the subgroup
 $G^0$ of $G$ is given by
\[ \cc^{\times}=\left\{\begin{pmatrix} a & -b & & & & \\ b & a & & & & \\ & & d & & & \\ & & & d & & \\ & & & & ad & -bd
\\& & & & bd & ad \end{pmatrix}:
a,b\in\rr,\, d=a^2+b^2\neq 0 \right\} . \] We note that $(0,1)$
and $(0,-1)$ are fixed points by the action of $\cc^{\times}$. If
$(s,t),(s',t')\in\mathcal{Z}-\{(0,\pm 1)\}$, $(s,t)\neq (s',t')$
and $\phi\cdot(s,t)=(s',t')$, with $\phi\in \cc^{\times}$, then by
writing down the equations, it can be shown that $tt'>0$ and both
$(s,t)$ and $(s',t')$ lie in the circle
\begin{equation}\label{circle} S_c:=\left \{(u,v)\in\mathcal{Z}:\,
u^2+\left(v-\frac{c}{2}\right)^2=\left(\frac{c}{2}\right)^2-1\right\}, \quad
c=t+\dfrac{1+s^2}{t}=t'+\dfrac{1+s'^2}{t'}.\end{equation}
 Each of these
circles intersects the $v$-axis in points $(0,t_0)$ and
$(0,1/t_0)$, with $0<|t_0|<1$, where $t_0= \dfrac{1}2 \left(c\mp\sqrt{c^2-1}\right)$, depending on the sign of $t$.
We show next that each $S_c$
coincides with an orbit $\mathcal
O_{0,t_0}$ of a point under the action of
$\cc^{\times}$. Since $\cc^{\times}\cong \rr_+\times S^1$ as Lie
groups, and the action of $\rr_+$ on $\mathcal{Z}$ is trivial, we
only have to consider the action of $S^1$ on $\mathcal{Z}$. The isotropy group of this
action is $\{ \pm 1\}$ and therefore each orbit is homeomorphic to $\mathbb{R} P^1 \cong S^1$.
Furthermore, each orbit is contained in one circle $S_c$ for some
$c$, hence, by using  topological arguments, we have that each $S_c$ coincides with an
orbit $\mathcal
O_{0,t_0}$ (see Figure~\ref{fig:orbits}), so that,
\[ \mathcal{C}_a(\fn_7)/\Aut(\fn_7)\cong \{J_{0,t}: \, 0<|t|\leq
1\}.\]
We note that the topology of the quotient space coincides with the relative topology of $\rr$, in particular, the moduli space of abelian complex structures on $\fn_7$ is disconnected.

\

$\bullet$ $\fn_4$: The group $G^0$ is given by a disjoint union
$G^0=G^+ \dot \cup G^-$, where
\[     G^+= \left\{  \begin{pmatrix}a&-\epsilon b & & & & \\   b&\epsilon a & & & &  \\
 & & a&-\epsilon b & &\\& & b&\epsilon a & & \\ & & & &a^2-b^2  &-2
 \epsilon ab  \\ & & & &2 ab &\epsilon (a^2-b^2)
\end{pmatrix} \, : \, \begin{array}{l} a, b\in \mathbb R, \\ a^2 +b^2 \neq 0,\\ \epsilon =\pm 1 \end{array}    \right\}, \;
G^-=\begin{pmatrix} & -I & & \\ I & & & \\ & & 1 & 0 \\ & &
0 & 1 \end{pmatrix}\cdot G_1^+. \] We note that the $G^0$-orbit
of $(0,1)$  has  two points: $(0, \pm 1)$. The orbit $\mathcal
O_{s,t}$ of $(s,t) \in\mathcal{Z}_1-\{(0,\pm 1)\}$ is given by the
disjoint union of two  circles (see Figure~\ref{fig:orbits}):
\begin{equation}\label{orbit} \mathcal
O_{s,t} :=\left \{(u,v)\in\mathcal{Z}_1:\, u^2+\left(v\pm
\frac{c}{2}\right)^2=\left(\frac{c}{2}\right)^2-1\right\},\end{equation} with
$c=t+\dfrac{1+s^2}{t}$. Each of these orbits intersects the
positive $v$-axis in points $(0,t_0)$ and $(0,1/t_0)$, with
$0<t_0<1$. Thus, $\mathcal Z / G ^0$ is parameterized by
$(0,1]$.


\begin{figure}[htbp] \label{fig:orbits}

 \centering
    \includegraphics{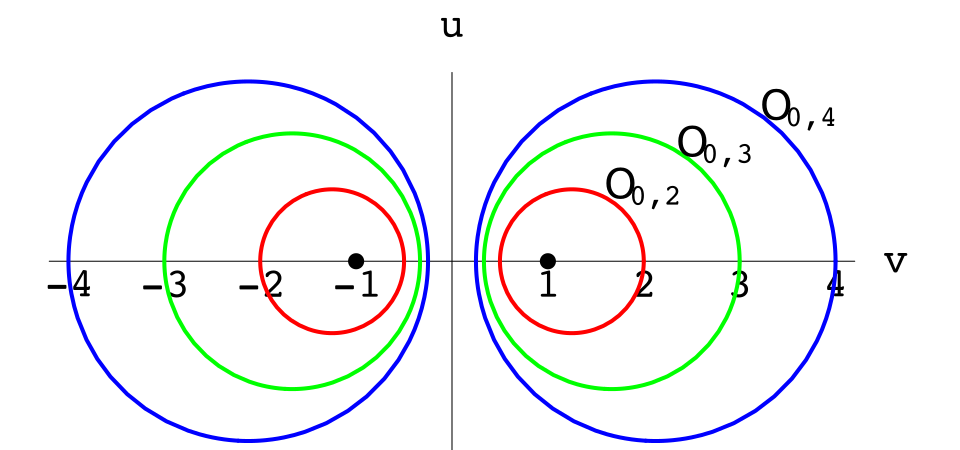}
		\caption{Orbits in $\mathcal{C}_a(\fn _7)$}
  \end{figure}

\

\noindent
The results in this section  are summarized in the next theorem.

\begin{teo}\label{todo} Let $\fn$ be a $6$-dimensional  nilpotent Lie
algebra with an abelian complex structure $J$. Then  $(\fn , J)$ is
holomorphically isomorphic to one and only one of the following:
\begin{enumerate}
\item $(\fn_1 ,J)$, with its unique complex structure: $J e_1=e_2, \;
J e_3=e_4, \; J e_5= e_6$,
\item $(\fn_2, J_\pm),$  with $J_\pm \, e_1=e_2, \;
J_\pm\, e_3= \pm e_4, \; J_\pm\, e_5= e_6$,
\item $(\fn _3 , J_{s}),$ with  $J_{s}\, e_1=e_2, \;
J_{s}\, e_3=e_4, \; J_{s}\, e_5=s e_5+ e_6, \;
s\in\rr$,
\item $(\fn _4 , J_{t})$ with  $J_{t}\, e_1=e_3, \;
J_{t}\, e_2=e_4, \; J_{t}\, e_5= te_6, \; t\in (0,1]$,
\item $(\fn _5 , J)$ with $J e_1=e_2, \;
J e_3=-e_4, \; J e_5= e_6$,
\item $(\fn_6 , J)$,
with  $J e_1=e_2, \; Je_3=-e_6, \; Je_4=e_5$,
\item $(\fn_7, J_{t})$ with $J_{t} e_1= e_2, \; J_{t} e_3=-e_4, \; J_{t}
e_5=te_6$, $0< |t| \leq 1$.
\end{enumerate}
\end{teo}

\medskip

\begin{rem}{\rm  There is another one-parameter family of abelian complex structures on $\fn_4$, given by \[ J^2_te_1=e_2,\quad J^2_te_3=-e_4, \quad J^2_te_5= te_6. \] In a previous version of this article, we stated that the abelian complex structures $J_t$ and $J^2_t, \, t\in (0,1]$, were not equivalent, but this statement was incorrect. E. Rodr\'iguez Valencia  provides in \cite{ERV} an automorphism of $\fn_4$ which shows the equivalence between these structures.}\end{rem}

\medskip

\begin{rem}{\rm  In \cite{L}  curves of non-equivalent abelian complex structures on the $6$-dimensional $2$-step nilpotent Lie algebras ${\mathfrak n}_3$ and ${\mathfrak n}_4$ were obtained by using geometric invariant theory. } \end{rem}

\medskip

\begin{rem} {\rm We observe that $\fn_4$ has a bi-invariant complex structure given by
\[      J_0 e_1 = e_2, \quad J_0 e_3=e_4, \quad J_0 e_5=e_6.       \]
According to \cite{ks} $J_0$ is the unique bi-invariant complex structure on $\fn_4$ compatible with the standard orientation
(see also \cite{ags}).
 } \end{rem}
 
 \medskip
 
 \begin{rem} {\rm It was proved in  \cite{u} that every complex structure on $\fn_6$
is abelian. This fact, together with Theorem \ref{todo}, implies
that $\fn_6$  admits a unique complex structure up to equivalence.
The existence of a unique abelian complex structure on $\fn_6$ was
already proved in \cite[Section 5.3]{CPoon}, where this Lie
algebra was denoted $\mathfrak{h}_9$.}
\end{rem}

\medskip

\textit{Notation}. The  Lie algebras in Theorem \ref{12345}
appear in the literature with different notations. In \cite{cfgu},
the nilpotent Lie algebras equipped with the so-called nilpotent complex
structures were denoted  by $\fh_k,\, k=1,\ldots, 16$. Salamon in
\cite{ss} associates to each Lie  algebra a $6$-tuple encoding the
expression of the differential of $1$-forms. We give next the
correspondence among the various notations:
\begin{enumerate}
\item[ ] $\fn_1 \longleftrightarrow \fh_8 \longleftrightarrow  (0,0,0,0,0,12)$,
\item[ ] $\fn_2 \longleftrightarrow  \fh_3 \longleftrightarrow (0,0,0,0,0,12+34)$,
\item[ ] $\fn_3 \longleftrightarrow \fh_2 \longleftrightarrow (0,0,0,0,12,34),$
\item[ ] $\fn_4 \longleftrightarrow \fh_5 \longleftrightarrow (0,0,0,0,13+42,14+23),$
\item[ ] $\fn_5 \longleftrightarrow  \fh_4 \longleftrightarrow (0,0,0,0,12,14+23),$
\item[ ] $\fn_6 \longleftrightarrow \fh_9 \longleftrightarrow
(0,0,0,0,12,14+25),$
\item[ ] $\fn_7 \longleftrightarrow \fh_{15} \longleftrightarrow (0,0,0,12,13+42,14+23)$.
\end{enumerate}

\

\section{Dimension six: the general case}\label{class}

 The classification, up to equivalence, of the $6$-dimensional
non-nilpotent Lie algebras $\fs$ admitting abelian complex
structures $J$ will be done in several steps, according to the
different possibilities for $\fs'_J$. When $J$ is proper it follows from \S\ref{four}
that  the ideal $\fs'_J$ with its induced
abelian complex structure coincides, up to equivalence, with
$\rr^2, \; \aff(\rr),\; \rr^4, \; \fh_3\times \rr,\; \aff(\rr)
\times \rr^2$, $ \aff(\rr) \times \aff(\rr)$, $ \aff(\rr)
\ltimes_{\ad} \rr^2$  endowed with their unique abelian complex
structure or with $(\aff(\cc),J_1), \; (\aff(\cc),J_2)$. Note that
according to Proposition \ref{s'_abel}, the case
$\fs'_J\cong\fh_3\times \rr$ does not occur.  In order to carry out the
classification we will consider  three different cases: $\dim
\fs'_J=2$,  $\dim \fs'_J=4$ or $\fs'_J =\fs$.

\subsection{$ \dim\, \fs'_J=2 $}\label{dim2}

\begin{teo}\label{teo_dim2} If $\fs$ is a $6$-dimensional non-nilpotent Lie algebra admitting an
abelian complex structure $J$ such that $\dim \fs '_J =2$, then
 $(\fs,J)$
is holomorphically isomorphic  to one of the following:
\begin{itemize}
\item $\aff(\rr) \times \rr ^4$ with its unique complex structure,
\item $\aff(\cc) \times \rr^2$ equipped
with the product $(J_1\times J)$, $J_1$ from Theorem
\ref{dim4}.
\end{itemize}
\end{teo}

\begin{proof}
Since $ {\rm dim}\, \fs'_J=2$, we have two possibilities:
\begin{enumerate}
\item ${\rm dim}\, \fs'=1$.  Since $\fs$ is non-nilpotent, $\fs$ is isomorphic to
$\aff (\rr) \times \rr ^4$ (see, for instance, \cite[Theorem
4.1]{bd}). According to Proposition \ref{s'=1}, this Lie algebra
has a unique complex structure, up to equivalence.
\item ${\rm dim}\, \fs ' = 2$, hence $\fs '=\fs'_J$ is $J$-stable.
Let $\fs'$ be spanned by $e_1,\,e_2$ and $Je_1=e_2$. Let
$\mathfrak u$ be a complementary $J$-stable subspace of $\mathfrak
s$, $\fs= \fs' \oplus \mathfrak u.$ From Lemma \ref{propiedades}$(iii)$,
the kernel of the linear map $\rho: \fu \to {\rm End}(\fs')$, $u
\mapsto \ad_{u}|_{\fs'}$ is $J$-stable. The Jacobi identity
implies that ${\rm Im} \rho$ is a commutative family of
endomorphisms of $\fs'$, therefore ${\rm dim}({\rm Im} \rho)<4$,
thus   ${\rm dim}(\ker \rho)$ is $4$ or $2$. When ${\rm dim}(\ker
\rho)=4$,  $\fs$ is $2$-step nilpotent, contradicting the
assumption on $\fs$. Indeed, if $\rho \equiv 0$, then $[x,u]=0$
for all $x\in\fs'$ and $u\in\fu$. Since $\fs'$ is abelian, we
obtain that $\fs'$ coincides with the center of $\fs$ and the
claim follows. Therefore, we must have ${\rm dim}(\ker \rho)=2$.
We show next that $(\fs,J)$ is holomorphically isomorphic to
$\aff(\cc)\times \rr^2$ equipped with the product $(J_1\times J)$
from Theorem \ref{dim4}. Indeed, let $f_1,f_2=Jf_1, f_3,
f_4=Jf_3$ be a basis of $\fu$ with $f_1,\,f_2\in\ker \rho$. Using
the Jacobi identity one obtains $[[f_1,f_2],
f_3]=[[f_1,f_2],Jf_3]=0$, therefore, since $J$ is abelian,
$[f_3,[f_1,f_2]]=[f_3,J[f_1,f_2]]=0.$ Thus, $[f_1,f_2]=0$, since
otherwise $f_3\in \ker \rho$, which is a contradiction.  Since
$\rho(f_3)\rho(f_4)=\rho(f_4)\rho(f_3)$ and $\rho(f_4)
=-\rho(f_3)J$ one has
\[ \rho(f_3)= \begin{pmatrix} x & -y   \\ y & x  \\
\end{pmatrix}  , \]
with $x^2+y^2 \neq 0$. We note that
$\fk:=\text{span}\{f_3,f_4,e_1,e_2\}$ is a $J$-stable Lie
subalgebra of $\fs$, and using Lemma \ref{iso_aff} we obtain that
$(\fk,J|_{\fk})$ is holomorphically isomorphic to
$(\aff(\cc),J_1)$, so that we may assume
\begin{gather*} [f_3,f_4]=[e_1,e_2]=0, \quad [f_3,e_1]=-[f_4,e_2]=e_1,\quad
[f_3,e_2]=[f_4,e_1]=e_2, \\ 
Jf_3=-f_4, \qquad Je_1 =e_2.
\end{gather*}
The remaining brackets are given by
\begin{align*} [f_1,f_3]&= ae_1 + be_2= -[f_2,f_4], \\ [f_1,f_4]&=
-be_1 + ae_2=[f_2,f_3],\end{align*} where the last line is a
consequence of $[[f_2,f_3],f_4]+[[f_4,f_2],f_3]=0$ and $J$
abelian.
If $a=b=0$, then $\fs=\aff(\cc)\times\rr^2$. If $a^2+b^2\neq 0$,
one considers
\[ \tilde {f}_1 = e_1+ \frac 1{a^2+b^2}(af_1-bf_2), \qquad \qquad  \tilde {f}_2 =J\tilde {f}_1= e_2+\frac 1{a^2+b^2}(bf_1+af_2),         \]
and finds that $\tilde {f}_1,\, \tilde {f}_2$ are central elements
in $\fs$. Therefore $\fs \cong \aff(\cc)\times\rr^2$ equipped with the complex structure given in the statement. \qedhere
\end{enumerate}
\end{proof}

\subsection{ $ \dim\, \fs'_J=4 $}\label{sec_4.4}

\

\noindent Let $\fs$ be a $6$-dimensional solvable Lie algebra
equipped with an abelian complex structure $J$ such that
$\dim\fs'_J=4$. Let $f_1,\, f_2=Jf_1$ span a complementary subspace
of $\fs'_J$, and set $D:=\ad_{f_1}|_{\fs_J'}$. It then follows from
Lemma \ref{propiedades}$(ii)$ that $\ad_{f_2}|_{\fs_J'}=-DJ$;
therefore $D$ and $DJ$ are both derivations of $\fs'_J$. The bracket
$[f_1,f_2]$ lies in $\fs'\subseteq\fs'_J$, and using the Jacobi
identity for $f_1,\, f_2$ and any $x\in\fs'_J$, we obtain that
\begin{equation}\label{DJ} \ad_{[f_1,f_2] }|_{\fs'_J}= DJD -D^2J.
\end{equation}

In what follows we analyze separately two different cases,
corresponding to $\fs'_J$ being abelian or non-abelian.

\medskip

\begin{teo}\label{non_abelian}
Let $\fs$ be a $6$-dimensional Lie algebra with an abelian complex
structure $J$ such that $\fs'_J$ is $4$-dimensional and non-abelian.
Then $(\fs , J)$ is holomorphically isomorphic to one of the
following:
\begin{enumerate} \item $\aff(\rr)\times(\fh_3\times\rr)$, with the
unique complex structure on each factor,
\item $\aff(\rr)\times\aff(\cc)$, where the complex structure on
$\aff(\cc)$ is given by $J_1$ (Theorem~\ref{dim4}),
\item $ \aff(\rr)\times \aff (\rr) \times \rr^2$, where each factor
carries its unique  complex structure,
 \item $
(\aff(\rr) \ltimes_{\ad} \rr^2)\times \rr^2$, where each factor is equipped with
its unique abelian complex structure,
 \item $\fa \ff\ff(\cc)\times \rr^2 $, where the first factor is equipped
with $J_2$ from Theorem~\ref{dim4}.
\end{enumerate}
\end{teo}

\begin{proof}
We will consider several cases, according to the isomorphism class of $\fs'_J$.

\medskip

$(i)$ Case $\fs'_J\cong\aff(\rr)\times\rr^2$:
There exists a basis $\{e_1, \ldots, e_4\}$ of $\fs'_J$ such that
$[e_1,e_2]=e_2$ and $J$ is given by $Je_1=e_2, \; Je_3=e_4$. A
derivation $D$ of $\fs'_J$ such that $DJ$ is also a derivation
takes the following form:
\[ D= \begin{pmatrix} 0 & 0 & &  \\ a &b& & \\ & & c&e \\ & &d &f
\end{pmatrix} , \]
with $a,b,...,f \in \rr $. From \eqref{DJ}, we obtain that \[
e=-d, \quad f=c, \quad [f_1,f_2]= (a^2+b^2)e_2+u, \quad u\in
\text{span}\, \{e_3,e_4\}.  \]
Since $\fs'_J= \aff(\rr) \times \rr^2$, it follows that $c^2+d^2\neq 0$ or $u\neq 0$.

We show next that we may assume $a=b=0$. Indeed, if $a^2+b^2\neq
0$, consider another complementary subspace of $\fs'_J$ spanned by
\[  \tilde{f}_1= e_1-\frac{bf_1+af_2}{a^2+b^2}, \qquad \quad  \tilde{f}_2= J\tilde{f}_1= e_2+\frac{af_1-bf_2}{a^2+b^2}.\]
The only non-vanishing brackets with respect to the basis above are
\[ [e_1,e_2]=e_2,\, [\tilde{f}_1,\tilde{f}_2]=\tilde{u}, \,
[\tilde{f}_1,e_3]=[\tilde{f}_2,e_4]=ce_3+de_4,\,
[\tilde{f}_1,e_4]=-[\tilde{f}_2,e_3]=-de_3+ce_4,\] with
$\tilde{u}\in\text{span}\{e_3,e_4\}$, for some new values of $c$ and $d$.

Note that $\fs$ can be decomposed as $\text{span}\{e_1,e_2\}\oplus
\text{span}\{\tilde{f}_1,\tilde{f}_2,e_3,e_4\}$. If $c=d=0$, then
$u\neq 0$ and there is a holomorphic isomorphism between
$\text{span}\{\tilde{f}_1,\tilde{f}_2,e_3,e_4\}$ and $\fh_3\times
\rr$, so that $\fs\cong \aff(\rr)\times(\fh_3\times \rr)$. If
$c^2+d^2\neq 0$, using Lemma \ref{iso_aff} we obtain that
$\text{span}\{\tilde{f}_1,\tilde{f}_2,e_3,e_4\}$ is holomorphically
isomorphic to $(\aff(\cc),\, J_1)$, so that, in this case, $\fs\cong
\aff(\rr)\times\aff(\cc)$.

\

$(ii)$ Case $\fs'_J\cong\aff(\rr)\times\aff(\rr)$:
There exists a basis $\{e_1, \ldots, e_4\}$ of $\fs'_J$ such that $[e_1,e_2]=e_2,\;
[e_3,e_4]= e_4$ and $J$ is given by $Je_1=e_2, \; Je_3=e_4$. A derivation $D$ of $ \fs'_J$ is of the form
\begin{align*}
D &= \begin{pmatrix} 0 & 0 & 0 & 0 \\ a & b & 0 & 0 \\ 0 & 0 & 0 & 0 \\ 0 & 0 & c & d \\
\end{pmatrix}
\end{align*}
with $ a,b,c,d\in \RR$, and we observe that $DJ$ is also a
derivation. From \eqref{DJ}, we obtain that
\[ [f_1,f_2]= (a^2+b^2)e_2+(c^2+d^2)e_4. \]
If $a^2+b^2=c^2+d^2=0$, we obtain that $\fs=\aff(\rr)\times\aff(\rr) \times \rr^2 $. If $a^2+b^2\neq
0, \; c^2+d^2=0$, then setting
\begin{equation}
\label{basis}\tilde{f}_1 = -e_1+\frac{bf_1+af_2}{a^2+b^2}, \qquad
\quad \tilde{f}_2 =J \tilde{f}_1= -e_2+\frac{-af_1+bf_2}{a^2+b^2},
\end{equation}
we obtain that $[\tilde{f}_1  , \tilde{f}_2 ]=0 $ and $\tilde{f}_1 ,\, \tilde{f}_2 $
commute with span$\{ e_1, \dots, e_4\}$. Therefore, $\fs \cong
\aff(\rr)\times\aff(\rr) \times \rr^2 $ via a holomorphic isomorphism. The proof for $a^2+b^2 = 0, \;
c^2+d^2\neq 0$ is analogous. Finally, when $a^2+b^2 \neq 0, \;
c^2+d^2\neq 0$, we define $\tilde{f}_1 ,\, \tilde{f}_2 $ as in
\eqref{basis} and this case reduces to the previous one.

\

$(iii)$ Case $\fs'_J\cong\aff(\rr)\ltimes _{\ad}\rr^2$:
There exists a basis $\{e_1, \ldots, e_4\}$ of $\fs'_J$ such that
$[e_1,e_2]=e_2, \; [e_1,e_4]=e_4,\; [e_2,e_3]=e_4$ and $Je_1=e_2,
\; Je_3= -e_4.$ A derivation $D$ of $\fs'_J$ such that $DJ$ is
also a derivation is of the form:
\[ D= \begin{pmatrix} 0 & 0 & 0 & 0 \\ d &c& 0& 0\\ 0&0 &0 &0  \\a &b
&-d &c
\end{pmatrix} , \]
with $a,b,c,d \in \rr $. Using \eqref{DJ}, we obtain that
\[ [f_1,f_2]= (c^2+d^2)e_2+ 2(ad+bc) e_4.\]

If $a^2+b^2=c^2+d^2=0$, then $\fs= (\aff(\rr)\ltimes _{\ad}\rr^2)\times \rr^2 $.
If $c^2+d^2\neq 0, \; a^2+b^2=0$, then setting
\begin{equation}
\label{basis3}\tilde{f}_1 = -e_1+\frac{cf_1+df_2}{c^2+d^2}, \qquad
\quad \tilde{f}_2 =J \tilde{f}_1= -e_2+\frac{-df_1+cf_2}{c^2+d^2},
\end{equation}
one has that $[\tilde{f}_1  , \tilde{f}_2 ]=0 $ and $\tilde{f}_1 ,\,
\tilde{f}_2 $ commute with span$\{ e_1, \dots , e_4\}$, and
therefore $\fs\cong(\aff(\rr)\ltimes _{\ad}\rr^2) \times \rr^2$ via a
holomorphic isomorphism. If $c^2+d^2 = 0, \; a^2+b^2 \neq 0$, we
have again the holomorphic ismorphism
$\fs\cong (\aff(\rr)\ltimes _{\ad}\rr^2) \times \rr^2$ by considering
the following change of basis:
\[ \tilde{f}_1 = f_1+be_3+ae_4, \qquad \quad
\tilde{f}_2 =J \tilde{f}_1= f_2+ae_3-be_4.
\]
Finally, when $c^2+d^2\neq 0, \; a^2+b^2\neq 0$, by applying \eqref{basis3} this case reduces
to the previous one.

\medskip

($iv$) Case $\fs'_J\cong  \fa \ff\ff(\cc)$: There exists a basis
$\{e_1, \ldots, e_4\}$ of $\fs'_J$ such that
\[[e_1,e_3]=-[e_2,e_4]=e_3,\quad [e_1,e_4]=[e_2,e_3]= e_4.\] It is
known that a derivation $D$ of $\aff(\cc)$ has matrix form
\[ D = \begin{pmatrix} 0 & 0 & 0 & 0 \\ 0 & 0 & 0 & 0 \\ c & -d & a & -b \\ d & c & b & a \\
\end{pmatrix} \] with respect to $e_1,\ldots, e_4$, for some
$a,b,c,d\in\rr$. From Theorem \ref{dim4}, we may assume that
the restriction of $J$ to $\fs '_J$ is  $J_1$ or $J_2$.

Let us consider first the case when the restriction is $J_1$. It
turns out that $DJ_1$ is also a derivation of $\fs'_{J_1}$, and
using \eqref{DJ}, we obtain that
\[ [f_1,f_2]= 2(ad+bc)e_3+ 2(bd -ac)e_4, \]
for some $ a,b,c,d \in \RR$. It follows that ${\rm dim}\, \fs'_{J}
= 2$, and therefore, this case cannot occur.

If the restriction is $J_2$, we have that $DJ_2$ is also a
derivation of $\fs'_{J_2}$, and using \eqref{DJ}, we obtain that
\[ [f_1,f_2]= ((a^2 - b^2) +(c^2 - d^2)) e_3+ 2(ab +cd)e_4 ,\]
with $ a,b,c,d \in \RR$.

If $a^2+b^2=c^2+d^2=0$, then $\fs \cong \aff(\cc) \times \rr^2 $. If
$c^2+d^2\neq 0, \; a^2+b^2=0$, then setting \begin{equation}
\label{basis1}\tilde{f}_1 = f_1+ce_3+de_4, \qquad \quad \tilde{f}_2
=J \tilde{f}_1= f_2-ce_1-de_2,
\end{equation} we obtain that $[\tilde{f}_1  , \tilde{f}_2 ]=0 $
and $\tilde{f}_1 ,\, \tilde{f}_2 $ commute with span$\{ e_1, \dots ,
e_4\}$. Therefore, $\fs \cong  \aff(\cc) \times \rr^2 $ via a
holomorphic isomorphism. The proof for  $c^2+d^2= 0, \; a^2+b^2\neq
0$ is analogous. Finally, when $c^2+d^2\neq 0, \; a^2+b^2\neq0$, we
define $\tilde{f}_1 ,\, \tilde{f}_2 $ as in \eqref{basis1} and this
case reduces to the previous one.
\end{proof}

\medskip

We prove next some properties of $(\fs, J)$ for a $6$-dimensional non-nilpotent Lie algebra $\fs$ with an abelian complex
structure $J$ such that $\fs'_J\cong \rr^4$. Given $x\notin \fs'_J$, we denote by $D_x: \fs'_J \to \fs'_J$  the restriction of $\ad_x$ to $\fs'_J$.

\begin{lem}\label{claims} Let $\fs$ be a $6$-dimensional non-nilpotent Lie algebra with center $\fz$ (possibly trivial) and denote $\fs^2=[\fs, \fs ']$. If $J$ is
an abelian complex
structure on $\fs$ such that $\fs'_J\cong \rr^4$, then
\begin{enumerate}[(i)]
\item \label{01}  $D_x$ commutes with $D_xJ$ for any $x\notin \fs'_J$. 
\item \label{02} {\rm Im}$D_x = \fs^2$ for any $x\notin \fs'_J$. 
\item \label{03} $\dim \fs ^2=2$ or $4$. 
\item \label{035} $\ker D_x = \fz$ for any $x\notin \fs'_J$. 
\item \label{04} $\fs'_J=\fs^2 \oplus \fz$. 
\item \label{05} There exists $x\notin \fs'_J$ such that $[x,Jx]\in \fz$. 
\end{enumerate}

\end{lem}
\begin{proof}
\eqref{01} Since $J$ is abelian, $D_{Jx}=-D_x J$. Using that $\fs'_J$ is abelian and the Jacobi identity (see \eqref{DJ}), \eqref{01}
follows.

 \eqref{02} Let $x\notin
 \fs'_J$ and set $V:=\text{span} \{ x, Jx   \}$. We have that $\fs= V\oplus \fs'_J$, therefore
 \[ \fs ^2=[\fs , \fs ']= [V, \fs '] = D_x \fs ' +D_{Jx}\fs '=  D_x \fs ' +D_{x}J\fs '=D_x
 \fs'_J .\]

  \eqref{03} Let $x\notin
 \fs'_J$,  then it follows from \eqref{02} that \eqref{03} is
 equivalent to $\dim \left( \text{Im}\, D_x \right)=2$ or $4$. We note first that
 $\text{Im}\, D_x=\{0\}$ is not
possible because this would give $\dim \fs'_J\leq 2$.

If $\dim \left(\text{Im}\, D_x \right)=1$, let $\text{Im}\, D_x = \rr y$. There
exist $a, b\in \rr$ such that $D_xy=ay, \; D_x Jy=by$. From
\eqref{01} $D_x^2=-(D_xJ)^2$, and this implies  $a=b=0$.
Therefore,  $D_x^2=D_{Jx}^2=0$ and $\fs$ is nilpotent, which is
not possible since $\fs$ is non-nilpotent by assumption.

If $\dim \left(\text{Im}\, D_x \right)= 3$, then $\dim \left( \text{ker} D_x\right) =1$, that
is, $\text{ker} D_x =\rr y$. From \eqref{01} we have
$D_x(D_xJy)=D_xJ(D_xy)=0$, hence $D_xJy =cy$. Using again that
$D_x^2=-(D_xJ)^2$, we obtain $0=D_x^2y=-(D_xJ)^2y =-c^2y$, hence
$c=0$. This implies that $Jy \in \text{ker} D_x=\rr y$, a
contradiction, and \eqref{03} follows.

\eqref{035} It follows from \eqref{02} and \eqref{03} that $\fz
\subset \fs'_J$. Moreover,
\begin{equation} \label{centro} \fz =
\ker D_x \cap J\left( \ker D_{x}\right), \qquad  \text{ for any } x\notin \fs'_J .
\end{equation}
If $\dim \fs^2=4$, then using \eqref{02} we obtain
that $\ker D_x =\{ 0\}$ for any $x\notin \fs'_J$, hence $\fz = \{
0\}$.

If $\dim \fs^2=2$, then $\dim \left( \ker D_x\right) =2 $ for any $x\notin
\fs'_J$. Let us assume that there exists $x\notin \fs'_J$ such
 that $\ker D_x$ is not $J$-stable, therefore, $\fs'_J= \ker D_x \oplus J\left( \ker
 D_x\right)$. Using \eqref{02} we obtain that $D_x : J\left( \ker
 D_x \right) \to \fs ^2$ is an isomorphism, then
 \[
 D_x \left( \fs ^2 \right) = D_x D_x J\left(  \ker  D_x \right) =  D_x J D_x \left(  \ker  D_x \right) =0 .
 \]
 This implies $\fs^2= \ker  D_x$, which means that
 $D_x^2=0=D_{Jx}^2$, contradicting the fact that $\fs$ is not
 nilpotent. Hence, $\ker  D_x$ is $J$-stable for any $x\notin
 \fs'_J$ and \eqref{centro} implies \eqref{035}.

 \eqref{04} If $\dim \fs^2=4$ there is nothing to prove (see proof of
 \eqref{035}). We consider next the case $\dim \fs^2=2$, which implies $\dim \fz
 =2$. We only need to  show that $ \fs ^2\cap \fz = \{0\}$.
 If $\fs ^2\cap \fz  \neq \{0\}$, we fix $x\notin \fs'_J$ and consider a $J$-stable complementary subspace $V$ of
 $\ker D_x=\fz$ in $\fs'_J$.  Using \eqref{02} we obtain that $D_x : V \to \fs ^2$ is an isomorphism. Let $v\in V, \; v\neq 0$, such that
 $D_xv \in \fs ^2\cap \fz$, then, using \eqref{01} and the fact that $\fz$ is $J$-stable, we have
 \[
 D_xD_xJv = D_xJD_xv =0,
 \]
 that is, $D_xJv \in \ker D_x=\fz$, thus $D_x ( V) =\fz$.  Therefore, $\fs^2= \fz$ and  $\fs$ is nilpotent, a
 contradiction, and \eqref{04} follows.

 \eqref{05}  Let $y\notin \fs'_J$, using \eqref{04} we have $[y,Jy]= D_y v +w$ with $w\in \fz$.
Setting $ x=y+\frac 12 Jv$, we compute
\begin{equation*} [x,Jx]= [y, Jy] -\frac 12 D_y v -\frac 12 D_{Jy}Jv =
 D_y v +w -D_y v  =w \in \fz ,
\end{equation*}
which implies \eqref{05}.

 \end{proof}

\begin{teo}\label{s'J=r^4}
Let $\fs$ be a $6$-dimensional non-nilpotent Lie algebra with an
abelian complex structure $J$ such that $\fs'_J\cong \rr^4$. Then
$(\fs ,J)$ is holomorphically isomorphic  to one and only one of
the following:

\smallskip

\begin{enumerate}
\item $(\aff(\cc)\times \rr^2, J):$ $[f_1,e_1]=[f_2,e_2]=e_1, \qquad
[f_1,e_2]=-[f_2,e_1]=e_2$, \\
\hspace*{.7cm} $Jf_1=f_2, \,Je_1=e_2-e_3,\, Je_2=-e_1-e_4,\, Je_3=-e_4$.

\smallskip

\item $(\fs_1, J_1):$  $[f_1,f_2]= e_3, [f_1,e_1]=[f_2,e_2]=e_1, \;\;
[f_1,e_2]=-[f_2,e_1]=e_2$,
  \\ \hspace*{.7cm} $Jf_1=f_2, \,Je_1=e_2-e_3,\, Je_2=-e_1-e_4,\, Je_3=-e_4$.

\smallskip

\item $(\fs_1,J_2) :$ $[f_1,f_2]=e_3, \;\;    [f_1,e_1]=[f_2,e_2]=e_1, \;\;
[f_1,e_2]=-[f_2,e_1]=e_2,    $
 \\ \hspace*{.7cm} $J_2f_1=f_2, \; J_2e_1=e_2,\; J_2e_3=e_4 $,

\smallskip

\item $(\fs_2 ,J):$
$    [f_1,e_1]=[f_2,e_2]=e_1,    \quad [f_1,e_2]=-[f_2,e_1]=e_2,$ \\
\hspace*{.8cm}
$[f_1,e_3]=[f_2,e_4]=e_1+e_3,  \quad [f_1,e_4]=-[f_2,e_3]=e_2+e_4 $, \\
\hspace*{.8cm} $Jf_1=f_2, \; Je_1=e_2,\; Je_3=e_4 $,

\smallskip

\item $(\fs_{(a,b)},J), \; (a,b) \in \mathcal R \, (\text{see } \eqref{region}):$

$\quad \, [f_1,e_1]=[f_2,e_2]=e_1,    \quad [f_1,e_2]=-[f_2,e_1]=e_2,$ \\
\hspace*{.8cm}
$[f_1,e_3]=[f_2,e_4]=ae_3+be_4,  \quad [f_1,e_4]=-[f_2,e_3]=-b e_3+ ae_4 $, \\
\hspace*{.8cm}
$Jf_1=f_2, \; Je_1=e_2,\; Je_3=e_4 $.
\end{enumerate}
\end{teo}

\begin{proof}
Let $ f_1\notin \fs'_J$ such that $[f_1,Jf_1]\in \fz$ and denote
$f_2=Jf_1$ (see Lemma \ref{claims}\eqref{05}), $D:=D_{f_1}$.

We consider next the different cases according to $\dim\fs^2$,
which can be $2$ or $4$ (see Lemma \ref{claims}\eqref{03}).

\medskip

$\bullet$ If $\dim\fs^2=2$, then there are two possibilities,
depending on $\fs^2$ being $J$-stable or not.

($i$) Case $\fs^2$ is $J$-stable: we note that $[f_1,f_2]\neq 0$,
since $\dim\fs'_J=4$. Let $e_1, e_2=Je_1$ be a basis of $\fs^2$
and $e_3=[f_1,f_2], \; e_4=Je_3$ a basis of $\fz$. We have
\[   D= \begin{pmatrix} A&0\\0&0 \end{pmatrix}, \quad \quad D_{f_2}=-DJ= \begin{pmatrix} -A J_0&0\\0&0 \end{pmatrix},\]
where $J_0=\begin{pmatrix} 0&-1\\1&0 \end{pmatrix}$ and $A$ is a
$2\times 2$ non-singular matrix satisfying $AJ_0=J_0A$, therefore,
$ A=\begin{pmatrix} a&-b\\b&a
\end{pmatrix}$ with $a^2+b^2\neq 0$.
Setting \begin{equation}\label{base} \tilde f_1 = \frac
1{a^2+b^2}(af_1+bf_2), \qquad \qquad  \tilde f_2 =J\tilde
f_1= \frac 1{a^2+b^2}(-bf_1+af_2),  \end{equation} we may
assume $A= I$. Note that $[ \tilde f_1, \tilde f_2]=\frac{1}{a^2+b^2}\, e_3$, therefore setting $\tilde e_3=\frac{1}{a^2+b^2}\, e_3$ and $\tilde e_4=J\tilde e_3$, the Lie bracket becomes:
\[ [\tilde f_1,\tilde f_2]=\tilde e_3, \qquad      [\tilde f_1,e_1]=[\tilde f_2,e_2]=e_1,    \qquad [\tilde f_1,e_2]=-[\tilde f_2,e_1]=e_2 ,  \]
which is the Lie algebra $\fs_1$ in the statement with the complex
structure $J_2$.

\medskip

($ii$) Case $\fs^2$ is not $J$-stable: fix $e_1\in\fs^2,\, e_1\neq
0$ such that $Je_1\notin \fs^2$. Due to Lemma \eqref{claims}(v) and the fact that $\fz$ is $J$-stable, we may write $Je_1=e_2-e_3$ with $e_2\in \fs^2$ and $e_3\in \fz$, with $e_j\neq 0$, $j=2,3$. It follows that $\{e_1,e_2\}$ is a basis of $\fs^2$. Let $e_4=-Je_3$; in the basis $\{e_1,\ldots, e_4\}$ of $\fs'_J$ we have that 
\[ D= \begin{pmatrix} A&0\\0&0 \end{pmatrix}, \quad \quad J= \begin{pmatrix} J_0&0\\-I&-J_0 \end{pmatrix}, \quad \quad D_{f_2}=-DJ= \begin{pmatrix} -AJ_0&0\\0&0 \end{pmatrix}\]
with $J_0$ as above. Again, $A$ is non-singular and
commutes with $J_0$, and making a change of basis analogous to $\eqref{base}$, we may assume that $A=I$. Then $J$ is given by
\begin{equation}\label{s1} Jf_1=f_2, \,Je_1=e_2-e_3,\, Je_2=-e_1-e_4,\,
Je_3=-e_4.\end{equation}
We have to analyze two cases, depending on $[f_1,f_2]$.

\smallskip

$\diamond$ If $[f_1,f_2]=0$, the Lie brackets of $\fs$ are given by
\[ [f_1,e_1]=[f_2,e_2]=e_1, \qquad [f_1,e_2]=-[f_2,e_1]=e_2 ,  \]
and $J$ is given as in \eqref{s1}. Therefore, $\fs\cong
 \aff(\cc) \times \rr^2 $ with a complex structure that preserves the
ideal $\rr^2$ but does not preserve the ideal $\aff(\cc)$.

\smallskip

$\diamond$ If $[f_1,f_2]=ae_3+be_4\in\fz$, with $a^2+b^2\neq 0$, let us set 
\[ \tilde e_1=ae_1+be_2, \quad \tilde e_2=-be_1+ae-2, \quad \tilde e_3=ae_1+be_2, \quad \tilde e_4=-be_1+ae-2,\]
Using this basis, it is clear that this Lie algebra is isomorphic to $\fs_1$ in Case ($i$) above, with a complex structure $J_1$ given by \eqref{s1}. We note that this complex structure is not equivalent to $J_2$ in Case ($i$) since $J_2$ preserves $\fs^2$ while $J_1$ does not.

\medskip

$\bullet$ If $\dim \fs ^2 =4$, then $\fz=0$, hence $[f_1,f_2]=0$,
$D$ is an isomorphism, and thus $D$ commutes with $J$. Therefore,
we obtain $\fs =\rr^2\ltimes \rr ^4$, where the complex structure
on $\rr^4=\text{span}\{e_1,\ldots, e_4\}$ is given by $Je_1=e_2,\,
Je_3=e_4$. Considering $D$ as an element of $GL(2,\cc)$, the
possible normal Jordan forms for $D$ are
\[ \begin{pmatrix} \alpha & 1 \\ 0 & \alpha
\end{pmatrix} , \quad \begin{pmatrix} \alpha & 0 \\ 0 & \beta\end{pmatrix},\]
with $\alpha,\,\beta\in\cc-\{0\},\, |\alpha|\geq |\beta|$.
Identifying $\rr^2=\text{span}\{f_1,f_2\}$ with
$\cc=\text{span}_{\cc}\{f_1\}$, we may take $\alpha^{-1}f_1$ and
therefore we may assume $\alpha=1, \, 0<|\beta|\leq 1$. Thinking
of $D$ now as an element of $GL(4,\rr)$, this gives rise to the
following possibilities for $D$, up to complex conjugation:
\begin{equation} \label{rho} D' = \begin{pmatrix} 1 & 0 & 1 & 0 \\ 0 & 1 & 0 & 1 \\  &  & 1 & 0 \\  &  & 0 & 1 \\
\end{pmatrix} ,  \quad \text{or} \quad
D_{(a,b)} = \begin{pmatrix} 1 & 0 &  &  \\ 0 & 1 &  &  \\  &  & a & -b \\  &  & b & a \\
\end{pmatrix}, \quad a, b \in \rr, \;\; 0<a^2+b^2 \leq 1 .
\end{equation}
Let us denote by $\fs _2$ (resp. $\fs_{(a,b)}$) the Lie algebra
determined by $D'$ (resp. $D_{(a,b)}$). Clearly, $\fs_2$ is not
isomorphic to $\fs_{(a,b)}$ for any $(a,b)$. Furthermore, it can
be shown that the isomorphism classes of the Lie algebras
$\fs_{(a,b)}$ are parameterized by \begin{equation}\label{region}
\mathcal R =\{(a,b)\in\rr^2\,:\, 0<a^2+b^2<1\}\cup \{(a,b)\in
S^1:\, b\geq 0\}. \qedhere
\end{equation}
\end{proof}

\medskip

\begin{rem} {\rm We point out that the Lie algebra $\fs_{(a,b)}$
from Theorem \ref{s'J=r^4} has a bi-invariant complex
structure given by $Jf_1=-f_2,\, Je_1=e_2, \; Je_3=e_4$. Indeed,
$\fs_{(a,b)}$ is the realification of the complex Lie algebra
$\mathfrak r_{3, \lambda}$ with $\lambda =a+ib$, where $\mathfrak
r_{3, \lambda}$  has a $\cc$-basis $\{w_1,w_2,w_3\}$ with Lie
brackets
\[    [w_1,w_2]=w_2, \qquad [w_1,w_3]= \lambda w_3. \]
The analogous statement holds for $\fs_2$, which is the
realification  of the complex Lie algebra $\mathfrak r_{3}$ with
$\cc$-basis $\{w_1,w_2,w_3\}$ and Lie brackets
\[    [w_1,w_2]=w_2, \qquad [w_1,w_3]= w_2 + w_3. \] }
\end{rem}

%
%

\

\subsection{$\fs '_J=\fs$} A family of Lie algebras with  abelian complex structure can be constructed in the following way.
Let $A$ be a commutative associative algebra  and set
$\aff (A):= A\oplus A$
with the following Lie bracket and standard complex structure $J$: \begin{equation}\label{j-aff}
[(x,y),(x',y')]=(0, xy'-x'y) ,  \qquad J(x,y)= (-y,x), \end{equation} then $J$
is  abelian (see \cite{ba-do}). Moreover, $J$ is non-proper if and only if $A^2=A$.

We show next that any $6$-dimensional  Lie algebra with a non-proper abelian complex structure is obtained using this procedure.

\begin{teo}\label{non-proper} Let $\fs$ be a $6$-dimensional  Lie algebra with a non-proper abelian complex structure~$J$.
Then
\begin{enumerate}[(i)]
\item $\dim \fs ' =3$ and $(\fs, J)$ is holomorphically isomorphic to
$\aff (A)$ with its standard complex structure, where $A$ is a $3$-dimensional commutative associative algebra such that $A^2=A$;
\item $(\fs, J)$ is holomorphically isomorphic to one and only one of the following:

\smallskip

\begin{enumerate}
\item[(1)] $\aff(\rr)\times \aff(\rr)\times \aff (\rr)$, where each factor
carries its unique  complex structure,
\item[(2)]  $\aff(\rr)\times \aff (\cc) $, where the second factor is equipped
with $J_2$ from Theorem~\ref{dim4},
 \item[(3)]  $\aff(\rr) \times
(\aff(\rr) \ltimes_{\ad} \rr^2)$, where each factor is equipped with
its unique abelian complex structure,
\item[(4)]  $(\fs_3 , J): \; [e_1, f_i]=f_i, \, i=1,2,3, \quad [e_2,f_1]=f_2, \quad [e_2,f_2]=f_3,\quad [e_3,f_1]=f_3,$\\
\hspace*{1.5cm} $Je_i=f_i, \; i=1,2,3$,
\item[(5)]   $(\fs_4 , J): \; [e_1, f_i]=f_i, \, i=1,2,3, \quad [e_2,f_1]=f_2, \quad  [e_3,f_1]=f_3,$\\
\hspace*{1.5cm} $Je_i=f_i, \; i=1,2,3$.
\end{enumerate} \end{enumerate}
\end{teo}

\begin{proof} $(i)$ Since $\dim \fs =6$ and $\fs _J'=\fs$,  it follows from (O2) in \S\ref{secaff} that $\dim \fs ' = 3$ or $4$. We show first that $\dim \fs' =4$ is not possible. If $\dim \fs '= 4$ then $\dim (\fs ' \cap J\fs ' ) =2$.
We note that $J$ abelian implies  $\fs ' \cap J\fs ' \subset \fz$ and $\fs$ has a basis of the form $\{ e_1,e_2,e_3, Je_1,Je_2,Je_3\}$ with $e_1, e_2 \in \fs ', \; e_3\in  \fs ' \cap J\fs '$, therefore $\fs '=\text{span}\{ [e_1,Je_1], [e_2,Je_2], [e_1,Je_2]=[e_2,Je_1] \}$, contradicting that $\dim \fs' =4$.

Thus $\dim \fs'=3$ and $\fs=\fs '\oplus J\fs '$. Note that
\begin{equation}\label{A^2=A}\fs '= [\fs ' , J\fs '].
\end{equation}
 The Lie bracket on $\fs$ induces a structure of commutative associative algebra on $\fs'$ in the following way
\[ x\cdot y= [x,Jy], \qquad  \text{for } x,y \in \fs '.
\]
Setting $A:= (\fs ', \cdot)$ it turns out that
\[ \phi: \aff (A) \to \fs , \qquad \phi (x,y)=y-Jx,\]
is a holomorphic isomorphism, where
 $\aff (A)$ is equipped with its standard complex structure \eqref{j-aff}. Moreover, \eqref{A^2=A} implies that $A^2=A$ and $(i)$ follows.

\smallskip

$(ii)$ According to $(i)$, in order to obtain the classification of the non-proper abelian complex structures
in dimension $6$, we must consider $\aff (A)$ with its standard complex structure, where $A$ is a $3$-dimensional commutative associative algebra satisfying $A^2=A$.
It was shown in \cite{gr} that there are five isomorphism classes of such associative algebras. These algebras can be realized as matrix algebras as follows:
\[ A_1=\left\{  \begin{pmatrix} a & & \\  &b& \\ & & c \end{pmatrix}    \right\},\qquad
A_2=\left\{  \begin{pmatrix} a & & \\  &b& -c\\ &c & b \end{pmatrix}     \right\}, \qquad
A_3=\left\{  \begin{pmatrix} a & & \\  &b&c \\ & & b \end{pmatrix}     \right\}, \]
\[
A_4=\left\{  \begin{pmatrix} a &b &c \\  &a& b\\ & & a \end{pmatrix}     \right\}, \qquad
A_5=\left\{  \begin{pmatrix} a & 0&c \\  &a&b \\ & & a \end{pmatrix}     \right\}, \]
with $a,b,c \in \rr$. The Lie brackets on $\aff(A_i), \; i=1, \dots , 5$, together with its standard complex structure are those given in the statement. It is easily verified that these Lie algebras are pairwise non-isomorphic, and this completes the proof of the theorem. \end{proof}

\medskip

\begin{rem} \label{nilradical}{\rm It is straightforward from Theorems \ref{teo_dim2},
\ref{non_abelian} and \ref{s'J=r^4}  that among the
non-nilpotent $6$-dimensional Lie algebras  carrying proper
abelian complex structures, $(\aff(\rr) \ltimes_{\ad}
\rr^2) \times \rr ^2 $ from Theorem \ref{non_abelian} is the unique Lie algebra whose nilradical is non-abelian.
Indeed, its nilradical  is $\fh _3 \times \rr^2$.
In the non-proper case, only $\aff(\rr)\times \aff(\rr)\times \aff (\rr)$ and  $\aff(\rr)\times \aff(\cc)$
have abelian nilradical  (see Theorem \ref{non-proper}). }
\end{rem}

\medskip

\begin{corol}
Let $\fs$ be a $6$-dimensional unimodular solvable Lie algebra
equipped with an abelian complex structure. Then $\fs$ is
isomorphic to $\fn_k,\, k=1,\ldots,7$, or to $\fs_{(-1,0)}$.
Moreover, the  simply connected Lie groups corresponding to these
Lie algebras admit compact quotients.
\end{corol}

\begin{proof}
The first assertion follows easily from Theorems \ref{12345},
\ref{teo_dim2}, \ref{non_abelian}, \ref{s'J=r^4} and
\ref{non-proper}. The simply connected Lie groups associated to
$\fn_k,\, k=1,\ldots, 7$, admit compact quotients since their
structure constants are rational (see \cite{M}). The existence of
compact quotients of the simply connected solvable Lie group with
Lie algebra $\fs_{(-1,0)}$ was proved in \cite{Y}.
\end{proof}

\

\

\end{document}